\documentclass[final,letterpaper]{siamart171218}
\usepackage{amssymb}
\usepackage{amsmath}
\usepackage{mathtools}
\usepackage{tabularx}
\usepackage{picture,slashbox}
\usepackage{graphicx}
\usepackage{subfigure}
\usepackage{color}
\usepackage{xcolor}
\usepackage{comment}

\usepackage{algorithm}
\usepackage{algorithmicx}
\usepackage{algpseudocode,tabularx,etoolbox}

\usepackage{url}
\usepackage{enumitem}

\usepackage[latin1]{inputenc}
\usepackage{caption} 
\graphicspath{{./figures/}}
\graphicspath{{./arxiv/}}

\usepackage{array,booktabs}
\usepackage{siunitx}
\DeclareMathOperator\erf{erf}

\newtheorem{remark}[theorem]{Remark}

\newcommand{\be}{\begin{equation}}
\newcommand{\ee}{\end{equation}}
\newcommand{\ba}{\begin{aligned}}
\newcommand{\ea}{\end{aligned}}
\newcommand{\bea}{\begin{eqnarray}}
\newcommand{\eea}{\end{eqnarray}}

\newcommand{\bc}{{\boldsymbol c}}
\newcommand{\bk}{{\boldsymbol k}}
\newcommand{\x}{{\boldsymbol x}}
\newcommand{\y}{{\boldsymbol y}}
\newcommand{\m}{{\boldsymbol m}}

\newcommand{\bphi}{{\boldsymbol{\phi}}}
\newcommand{\bpsi}{{\boldsymbol{\psi}}}
\makeatletter
\newcommand{\Scale}[2][1]{\scalebox{#1}{$\m@th#2$}}
\makeatother

\newcommand{\sbtr}{%
  \vcenter{\hbox{\Scale[0.65]{\blacktriangleright}}}%
}

\def\ccm{Center for Computational Mathematics, Flatiron Institute, Simons Foundation,
  New York, New York 10010}

\def\nyu{Courant Institute of Mathematical Sciences,
  New York University, New York, New York 10012}

\def\thu{Yau Mathematical Sciences Center,
  Tsinghua University, Beijing China 100084}

\def\papertitle{A new version of the adaptive fast Gauss transform for discrete and continuous
sources}

\title{\papertitle}
\author{Leslie Greengard%
  \thanks{\ccm\, \& \nyu\,
    ({\tt greengard@courant.nyu.edu}).}
  \and
  Shidong Jiang\thanks{\ccm\,
    ({\tt sjiang@flatironinstitute.org}).} 
  \and
  Manas Rachh\thanks{\ccm\,
    ({\tt mrachh@flatironinstitute.org}).} 
  \and
  Jun Wang\thanks{\thu\,
    ({\tt jwang2020@tsinghua.edu.cn}).} 
}

\makeatletter
\newcommand{\multiline}[1]{%
  \begin{tabularx}{\dimexpr\linewidth-\ALG@thistlm}[t]{@{}X@{}}
    #1
  \end{tabularx}
}
\makeatother

\AtBeginEnvironment{algorithmic}{\let\textbf\relax}

\algrenewcommand\algorithmicfor{For}
\algrenewcommand\algorithmicif{If}
\algblock{For}{End}
\algblock{If}{End}
\algrenewcommand{\algorithmiccomment}[1]{\hfill (#1)}

\begin{document}

\maketitle

\begin{abstract}
We present a new version of the fast Gauss transform (FGT) for discrete 
and continuous sources.
Classical Hermite expansions are avoided entirely, 
making use only of the plane-wave representation of the Gaussian kernel 
and a new hierarchical merging scheme. For continuous source 
distributions sampled on adaptive tensor-product grids, we exploit 
the separable structure of the Gaussian kernel to accelerate the 
computation.  For discrete sources, the 
scheme relies on the nonuniform fast Fourier transform 
(NUFFT) to construct near field plane wave representations. 
The scheme has been implemented for either free-space or periodic 
boundary conditions.  In many regimes, the speed is comparable to or 
better than that of the conventional FFT {\em in work per gridpoint}, 
despite being fully adaptive.
\end{abstract}



\pagestyle{myheadings}
\thispagestyle{plain}
\markboth{L. Greengard, S. Jiang, M. Rachh, and J. Wang}
{Adaptive fast Gauss transform}

\section{Introduction} \label{introduction}
In this paper, we consider the evaluation of the discrete and continuous Gauss transforms:
\be
u_i=\sum_{j=1}^N G(\x_i-\y_j;\delta)q_j,
\qquad \x_i, \y_j \in B, \qquad i=1,\ldots M,
\label{pgt}
\ee
and
\be
u(\x)=\int_B G(\x-\y;\delta)\sigma(\y)d\y,
\label{bgt}
\ee
where $B = [-\frac{1}{2},\frac{1}{2}]^d$ is the 
unit box centered at the origin in $\mathbb{R}^d$. For free space problems,
the Gaussian kernel is given by 
\be
G(\x;\delta)=e^{-\frac{\|\x\|^2}{\delta}} \,  ,
\label{gkernel}
\ee
while for periodic problems, 
\be
G(\x;\delta)= \sum_{{\bf j} \in {\mathbb{Z}^d}} e^{-\frac{\|\x + {\bf j}\|^2}{\delta}} \, ,
\label{gper_kernel}
\ee
where ${\mathbb{Z}}^d$ denotes the $d$-dimensional integer lattice.
This task is ubiquitous in computational physics, engineering, and statistics
and we do not seek to review the various applications here.

The original fast Gauss transform (FGT) \cite{greengard1991fgt} is 
a single level scheme 
that superimposes
on $B$ a uniform grid of boxes of side length $O(\sqrt{\delta})$ - 
ensuring that interactions
are negligible outside a modest range of near neighbors. 
More concretely, if the uniform grid spacing is exactly $\sqrt{\delta}$,
then after four neighbors in any direction, the Gaussian has decayed
to approximately $10^{-7}$ and an expansion of the field induced by 
the Gaussians in terms of Hermite functions requires an expansion
of about $10^d$ terms to achieve seven digit accuracy within the region
of interest. Using this ``low-rank representation",
the FGT is a linear scaling
algorithm, requiring only $O(M+N)$ work for the evaluation of \eqref{pgt}.

\begin{remark}
Reducing the number of boxes within the region of interest
can be accomplished by setting the box size to $4 \sqrt{\delta}$, so that
only near neighbors need be considered. However, the necessary 
expansion size then increases to more than $40^d$, and the algorithm
is less efficient.
\end{remark}

Since the creation of the original FGT, several improvements and 
extensions have been made (see, for example, 
\cite{elgammal2003,greengard1998nfgt,lee2006dtfgt,
sampath2010pfgt,spivak2010sisc,strain1994sisc,tausch2009sisc,
veerapaneni2008jcp,wang2018sisc}).
Of particular relevance are 
\cite{greengard1998nfgt,sampath2010pfgt,spivak2010sisc}, where 
Fourier-based methods were
introduced, providing efficient approximations which can be 
translated in diagonal form - yielding significant accelerations.
Also relevant is \cite{wang2018sisc}, which uses
both Hermite and Fourier-based expansions.
It is hierarchical, fully adaptive, 
insensitive to $\delta$, and able
to compute transforms with discrete sources,
volume sources and densities supported on 
boundaries. The algorithmic approach in 
\cite{wang2018sisc} follows that of the 
hierarchical fast multipole method (FMM) \cite{greengard1987jcp}, 
separating near and far field interactions.

Here, we introduce an alternate
scheme, which turns out to be both simpler and faster.
The new scheme eliminates the Hermite and local expansions that
have been central to previous FGTs. Instead, it relies only on the Fourier
(spectral) approximation of the Gaussian. 
For discrete sources, given a parameter $s$, an adaptive tree
is constructed by hierarchically dividing every box containing 
more than $s$ particles
into $2^d$ child boxes, until it reaches the so-called ``cutoff'' level 
$l_c$ with box dimension approximately 
$\sqrt{\log\left( \frac{3}{\epsilon} \right)} \, \sqrt{\delta}$
(discussed in \cref{sec:pfgt}). 
Briefly, the cutoff level is chosen so that only interactions between 
neighboring
boxes need to be considered for computing the potential to the specified
accuracy $\epsilon$.

The interactions between boxes at the cutoff level which 
contain more than $s$ sources and targets are handled using 
plane-wave expansions. This representation is
particularly efficient, since translation is diagonal in the plane wave 
basis (as they are eigenfunctions
of the translation operator). 
Since one only needs to consider the interactions between
nearest neighbors, the number of diagonal translations is minimal, 
especially when compared with the FMM-like FGT in \cite{wang2018sisc}.

A key observation used in the new scheme is that the formation and 
evaluation of these plane-wave interactions
can be accelerated by the nonuniform fast Fourier transforms (NUFFTs) 
\cite{nufft2,nufft3,nufft6,finufftlib} at a cost 
$O(N_F\log N_F + N_P)$, where $N_F$ is the number of discrete
Fourier modes and $N_P$ is the number of source and target points.
As noted in \cite{finufft}, 
the prefactor in front of $N_P$ is roughly $(D+1)^d$, where $D$ is the 
number of requested digits. Thus, the cost of forming and evaluating 
plane-wave expansions
is quite modest in terms of the number of required 
floating-point operations. 
The method described in this paper is unusual in that it uses 
the NUFFT in a spatially adaptive manner. 
We do {\em not} use the NUFFT to 
implement discrete convolution globally, rather to improve the constants
in the FGT framework.

For continuous sources, the adaptive tree is designed to
resolve the given density function $\sigma$,
and thus cannot be terminated at the cutoff level. 
While one could first apply a quadrature rule to the 
integral in \cref{bgt} and then call the new point FGT,
we can develop a much faster scheme by making use of the 
tensor product structure of the grid and the separable nature of the 
Gaussian kernel.  To be more precise,
integrating a multidimensional Gaussian against a product of 
polynomials $p_1(x_1) \dots p_d(x_d)$ can be done by building
one-dimensional tables of Gaussian convolution integrals for
a given $\delta$ and a fixed set of target points. 
Moreover, both forming and
evaluating plane-wave expansions can also be accelerated by separation of 
variables, since the input density is tabulated on a tensor product grid in 
each leaf box. Detailed analysis of the operation
count shows that the diagonal translations needed to merge and split 
the plane-wave expansions in the adaptive tree hierarchy are negligible
contributions to the overall computational cost. 
The systematic use of the tensor product structure of the continuous
FGT \cref{bgt} leads to 
\cref{alg4}, below, with single core performance of $5-20$ million
points per second for $d=2$ and 
$1-7$ million
points per second for $d=3$, depending on the precision.

An additional contribution of the present paper is an analysis of
the continuous transform \cref{bgt} in terms of {\em resolution} of the
density $\sigma(\x)$ and the output potential $u(\x)$.
While the former may be resolved on the adaptive input grid, it does {\em not}
follow that the latter is resolved on output. 
In general, it is necessary to build a new adaptive tree that is guaranteed
to resolve the potential.
We explain the reason for this in \cref{sec4.2} and describe an algorithm
that creates a suitable output tree.

The paper is organized as follows. We begin by reviewing the plane-wave approximation of
a Gaussian in \cref{sec:prelim}. We introduce the new discrete (point) FGT 
in \cref{sec:pfgt} and the new continuous (volume) FGT in \cref{sec:vfgt}. 
Periodic boundary conditions are discussed in \cref{sec:periodic}. Timing results are
presented in \cref{sec:numer}
and we conclude in \cref{sec:conclusions}
with a discussion of potential extensions of the present scheme.

\section{Mathematical preliminaries} \label{sec:prelim}

Since the Gaussian in higher dimensions is the 
product of one-dimensional Gaussians, we need only consider 
the one-dimensional case in detail. 
By virtue of the triangle inequality,
it is straightforward to see that the error in 
approximating the Gaussian
in $\mathbb{R}^d$ as the product of
one-dimensional approximations amplifies the error by 
a factor of roughly $d$.

\subsection{Spectral approximation of the Gaussian} \label{sec:spectral}

We will make use of the well-known Fourier transform relation
\be
G(x; \delta) = e^{-x^2/\delta}=\sqrt{\frac{\delta}{4\pi}}\int_{-\infty}^{\infty} e^{-k^2 \delta/4} e^{ikx}dk = \frac{1}{\sqrt{2\pi}} \int_{-\infty}^{\infty} \hat{G}(k;\delta) e^{ikx}dk \, .
\ee
Discretizing this integral with the trapezoidal rule yields an
approximation of the Gaussian by a finite sum of plane waves
\cite{greengard1998nfgt,spivak2010sisc}.
By using the fact that the effective range of $x$ is finite for any specified precision,
a modest number of quadrature nodes is sufficient to guarantee uniformly high accuracy.
This is made precise in the following lemma.

\begin{theorem}\label{thm:planewave}
Let $\varepsilon<0.1$ be a prescribed tolerance, and let us define a ``cutoff''
length $D_0$ by the formula
\be\label{D0}
D_0=\sqrt{\log\left(\frac{3}{\varepsilon}\right)}.
\ee
Then, for any $R\ge D_0$,
\be\label{pwerrorestimate}
\left|G(x;\delta) - \frac{h}{2\sqrt{\pi}}\sum_{m=-M}^{M-1} e^{-m^2 h^2/4}e^{imhx/\sqrt{\delta}}\right|
\le \varepsilon, \quad |x|\le R\sqrt{\delta},
\ee
where $M$ and the stepsize $h$ are given by the formulas
\be
h=\frac{2\pi}{R+D_0}, \qquad M=\left\lceil\frac{2D_0}{h}\right\rceil
=\left\lceil\frac{D_0(R+D_0)}{\pi}\right\rceil.
\label{pwterms}
\ee
\end{theorem}
\begin{proof}
From the Poisson summation formula, it follows that
\be
\ba
\sum_{n=-\infty}^{\infty}G\left(x+\frac{2\pi n \sqrt{\delta}}{h}; \delta \right)
&=\frac{h}{\sqrt{2\pi \delta}}\sum_{m=-\infty}^\infty\hat{G}\left(\frac{mh}{\sqrt{\delta}} \right)e^{imhx/\sqrt{\delta}}\\
&=\frac{h}{2\sqrt{\pi}}\sum_{m=-\infty}^\infty e^{-m^2h^2/4}e^{imhx/\sqrt{\delta}}.
\ea
\label{poisson}
\ee
Truncating the right-hand side of \cref{poisson} and rearranging terms,
we obtain a Fourier spectral approximation of the Gaussian:
\be
G(x; \delta) \approx \frac{h}{2\sqrt{\pi}}\sum_{m=-M}^{M-1} e^{-m^2 h^2/4}e^{imhx/\sqrt{\delta}},
\label{pwappr}
\ee
with error
\be
\left|G(x;\delta) - \frac{h}{2\sqrt{\pi}}\sum_{m=-M}^{M-1} e^{-m^2 h^2/4}e^{imhx/\sqrt{\delta}}\right|
\le E_T + E_A,
\ee
where the truncation error $E_T$ is given by
\be\label{truncerror}
E_T=  \frac{h}{2\sqrt{\pi}}\left|\sum_{m=-\infty}^{-(M+1)} e^{-m^2 h^2/4}e^{imhx/\sqrt{\delta}} + \sum_{m=M}^{\infty} e^{-m^2 h^2/4}e^{imhx/\sqrt{\delta}} \right|,
\ee
and the ``aliasing'' error $E_A$ is given by
\be\label{aliaserror}
E_A = \sum_{\substack{n=-\infty \\ n\ne 0}}^{\infty}G\left(x+\frac{2\pi n \sqrt{\delta}}{h}; \delta \right).
\ee
For the truncation error, we have
\be\label{truncerror2}
\ba
E_T
&\le \frac{h}{\sqrt{\pi}}\sum_{m=M}^{\infty} e^{-m^2 h^2/4}
= \frac{h}{\sqrt{\pi}}\sum_{k=0}^{\infty} e^{-(M+k)^2 h^2/4}\\
&= \frac{h}{\sqrt{\pi}} e^{-M^2 h^2/4}\sum_{k=0}^{\infty} e^{-(2kM+k^2) h^2/4}\\
&\le \frac{h}{\sqrt{\pi}} e^{-M^2 h^2/4}\sum_{k=0}^{\infty} e^{-2kM h^2/4}
= \frac{h}{\sqrt{\pi}} \frac{e^{-M^2 h^2/4}}{1-e^{-2M h^2/4}}.
\ea
\ee
For the aliasing error, under the assumption that $|x| \leq R \sqrt{\delta} < 2\pi \sqrt{\delta}/h$, we have
\be\label{aliaserror2}
\ba
E_A
&\le 2\sum_{n=1}^{\infty}G\left(\frac{2\pi n \sqrt{\delta}}{h}-R\sqrt{\delta}; \delta\right)\\
&=2 e^{-\left(\frac{2\pi}{h}-R\right)^2}\sum_{k=0}^{\infty}
e^{-2\frac{2k\pi}{h}\left(\frac{2\pi}{h}-R\right)-\left(\frac{2k\pi}{h}\right)^2}\\
&\le\frac{ 2 e^{-\left(\frac{2\pi}{h}-R\right)^2}}
{1-e^{-\frac{4\pi}{h}\left(\frac{2\pi}{h}-R\right)}}.
\ea
\ee
With $D_0$ given by \eqref{D0}, we have
\be
e^{-D_0^2}\le \frac{\varepsilon}{3}.
\ee
By \eqref{pwterms}, we obtain
\be
e^{-\left(\frac{2\pi}{h}-R\right)^2} = e^{-D_0^2}\le \varepsilon/3, \quad
e^{-M^2 h^2/4} = e^{-D_0^2} \le \varepsilon/3.
\ee
From \eqref{truncerror2}, \eqref{aliaserror2}, the assumption that
$\varepsilon\le 0.1$ and
the fact that $R\ge D_0$, it is straightforward to verify that
\be
E_T\le \varepsilon/3, \quad E_A\le 2\varepsilon/3.
\ee
The result \eqref{pwerrorestimate} follows.
\end{proof}
The total number of quadrature nodes needed in the Fourier spectral approximation, 
denoted by $n_f(R)$, is
simply equal to 
  \be
  n_f=2M = 2\left\lceil\frac{D_0(R+D_0)}{\pi}\right\rceil \, .
  \ee
In~\cref{thm:planewave}, we could have used one fewer term and truncated the sum to 
$[-M+1,M-1]$  instead. However, for numerical efficiency considerations, we
use an even number of quadrature nodes in the Fourier spectral approximation.
We have relied here on the trapezoidal rule, but the
midpoint rule would also achieve 
spectral accuracy.

\Cref{table1} lists the number of terms $n_f(R)$ for various values of  the 
precision $\varepsilon$ and the range $R$. 
For the new version of the point FGT, we can control box dimensions so that
the relevant value of $R$ is always $2D_0$. For the continuous FGT, 
where an adaptive tree is provided as input, we make use of the 
user-provided data structure 
and, with dyadic refinement, can only enforce
that $R$ falls in the interval $[2D_0,4D_0]$.

\begin{table}[t]
  \caption{\sf Number of terms needed in the Fourier spectral approximation
    of the Gaussian for the interval $|x|\le R\sqrt{\delta}$ with $R=2D_0$, $3D_0$
  and $4D_0$ as a function of the precision $\varepsilon$.}
\centering
\begin{tabular}{c|c|c|c|c}
\toprule
$\varepsilon$ & $D_0$ & $n_f(2D_0)$ & $n_f(3D_0)$ & $n_f(4D_0)$\\
\midrule
$10^{-2}$  &  2.39 &    12  &   16 & 20 \\
$10^{-4}$  &  3.21 &    20  &   28 & 34 \\
$10^{-6}$  &  3.86 &    30  &   38 & 48 \\
$10^{-8}$  &  4.42 &    38  &   50 & 64 \\
$10^{-10}$ &  4.91 &    48  &   62 & 78 \\
$10^{-12}$ &  5.36 &    56  &   74 & 92 \\
\bottomrule
\end{tabular}
\label{table1}
\end{table}

\subsection{The periodic kernel}

In order to impose periodic boundary conditions on a 
box of length $1$ centered at the origin,
we can replace the free-space Gaussian with an infinite set of images to obtain
the periodic kernel in the form
\be
G_p(x;\delta)=\sum_{m=-\infty}^\infty e^{-\frac{(x+m)^2}{\delta}} \, .
\label{gp:image}
\ee
Note that $\delta$ can be suitably rescaled if periodic conditions are to be imposed on 
a box of length $L$.
Using Fourier analysis, on the other hand, we can also write the periodic
kernel as
\be
G_p(x;\delta)=\sqrt{\pi \delta}
\sum_{n=-\infty}^\infty e^{-(\pi\sqrt{\delta}n)^2}e^{i2\pi nx}.
\label{gp:fourier}
\ee
The equivalence of the two formulas \eqref{gp:image} and \eqref{gp:fourier} can be proven
using the Poisson summation formula.
It is straightforward to verify that
the infinite series is well-approximated by a finite sum
to precision $\varepsilon$,
\be
G_p(x;\delta) = \sqrt{\pi \delta}
\sum_{n=-n_p}^{n_p} e^{-(\pi\sqrt{\delta}n)^2}e^{i2\pi nx} +O(\varepsilon),
\ee
with
\be
n_p=\left\lceil\frac{\sqrt{\log(1/\varepsilon)}}{\pi\sqrt{\delta}}\right\rceil.
\ee

\cref{table2} lists the values of $n_p$ for various values of $\varepsilon$
and $\delta$.
As $\delta \rightarrow  0$, the full Fourier series for the periodic
Gaussian on the unit interval requires more and more terms. 
However, once $\delta < 1/36$, the contribution of images beyond the 
first neighbors is approximately $10^{-16}$, so that only a
minor modification of the free-space FGT is needed for such cases 
(see \cref{sec:periodic}).

\begin{table}[t]
  \caption{\sf Number of terms needed in the Fourier series approximation
    of the periodic Gaussian on the unit interval.}
\centering
\begin{tabular}{c|c|c|c|c|c|c} 
\toprule
\backslashbox{$\delta$}{$\varepsilon$} & $10^{-2}$ & $10^{-4}$ & $10^{-6}$ & $10^{-8}$ & $10^{-10}$ &$10^{-12}$\\
\midrule
$10^{-1}$  &  3&4&4&5&5&6\\
$10^{-2}$  &  7&10&12&14&16&17\\ $10^{-3}$  & 22&31&38&44&49&53\\
\bottomrule
\end{tabular}
\label{table2}
\end{table}

\subsection{The non-uniform fast Fourier transform}
In order to construct plane-wave representations from a collection of Gaussian sources,
we will use the NUFFT
in $\mathbb{R}^d$.
For this,
let $n_{f}$ denote the number of desired Fourier modes in each
dimension $i = 1,\dots, d$, which we have assumed to be even.
Let
${\cal I}_{n_{f}} = \{ -n_{f}/2,\dots, n_{f}/2-1 \}$, then
${\cal I} = {\cal I}_{n_{f}} \times \dots \times {\cal I}_{n_{f}}$ contains
$N = (n_{f})^{d}$ entries. Given data points
$\{ \x_1,\dots,\x_M \}$, 
the type-1 NUFFT computes sums of the form
\be
F(\bk)=\sum_{j=1}^{M} f_j e^{\pm i \bk \cdot \x_j}, 
\ee
where $\bk \in {\cal I}$.
That is, it computes the discrete Fourier transform at equispaced modes from function values at
arbitrary data points.
It's adjoint, also called the type-2 NUFFT evaluates sums of the form
\be
f(\x_j)=\sum_{{\bf k} \in {\cal I}} F(\bk) e^{\pm i \bk \cdot \x_j}.
\ee
That is, it computes the value of a truncated Fourier series at
arbitrary target points.
The cost of both type-1 and type-2 NUFFTs is
$O(N\log N) + O(M)$.
We will make use of the FINUFFT library~\cite{finufftlib} in our implementation
and refer to the paper \cite{finufft} for further details and a historical review of the
algorithm itself.
The constant implicit in the notation $O(M)$ is roughly
$(D+1)^d$ where $D$ is the number of desired digits of accuracy and $d$ is the dimension
of the problem.

\subsection{The Fourier and Gauss transforms of Legendre and Chebyshev polynomials} \label{boxlemmas}

As we shall see below, the continuous FGT requires both the direct integration of the Gaussian kernel multiplied
by a smooth function and the Fourier transform of a smooth function when using the plane wave approximation from the preceding section. 
In both cases, the smooth function is approximated using orthogonal polynomials on $[-1,1]$.
In particular, we will make repeated use of the integrals
\be\label{gsintegrals}
J_\lambda[p_n](t) = \int_{-1}^1 e^{-(t-x)^2/\lambda^2} p_n(x) dx \, ,
\ee
and
\be\label{pwintegrals}
I[p_n](t) = \int_{-1}^1 e^{i t x} p_n(x) dx \, ,
\ee
where $p_n(x)$ is either the Legendre polynomial $P_n(x)$ or the Chebyshev polynomial $T_n(x)$
of degree $n$. 
While these could be computed numerically without a major loss of efficiency,
it is convenient to do so analytically (or semi-analytically).

A five-term recurrence 
for $J_\lambda[p_n]$ when $p_{n}$ are the Chebyshev polynomials was derived in~\cite{veerapaneni2008jcp}. 
For Legendre polynomials, the same approach leads to the following five-term recurrence,
\be\label{Jnrecurrence}
\ba
\frac{n+2}{2n+3}J_\lambda[P_{n+2}](t) &=
t\left(J_\lambda[P_{n+1}](t)-J_\lambda[P_{n-1}](t)\right)+\frac{(n-1)}{2n-1}J_\lambda[P_{n-2}](t)\\
&+(2n+1)\left(\frac{\lambda^2}{2}+\frac{1}{(2n+3)(2n-1)}\right)J_\lambda[P_n](t), \quad n\ge 4.
\ea
\ee
The first four terms in this sequence are given by
\be
\ba
J_\lambda[P_0] &= \frac{\sqrt{\pi}}{2} \left(\erf(b)-\erf(a)\right), \\
J_\lambda[P_1](t) &= -\frac{d_1}{2} + tJ_\lambda[P_0],\\
J_\lambda[P_2](t) &= \frac{3}{4}\left(-d_2-t d_1 + (\lambda^2-2/3+2 t^2)J_\lambda[P_0]\right), \\
J_\lambda[P_3](t) &= \frac{5}{8}\left(-d_1 + 2 t\left(4 J_\lambda[P_2](t)-J_\lambda[P_0]\right)/3
-(2/5-4\lambda^2)J_\lambda[P_1](t)\right),
\ea
\ee
where $\erf(x)=\frac{2}{\sqrt{\pi}}\int_0^x e^{-t^2}dt$ is the error function,
\be
a=\frac{-1-t}{\lambda}, \
b=\frac{1-t}{\lambda}, \ 
d_1=\lambda\left(e^{-b^2}-e^{-a^2}\right), \, \, \textrm{and} \, \,
d_2=\lambda\left(e^{-b^2}+e^{-a^2}\right).
\ee
Empirically, we observe that \eqref{Jnrecurrence} appears stable for $\lambda\le 1/8$, 
similar to the observation in 
\cite{veerapaneni2008jcp} for the case of Chebyshev polynomials.
For larger values of $\lambda$, the integrand is very smooth and easily computed.

As for $I[p_{n}]$, when $p_{n}$ are Legendre polynomials, we have~\cite[\S10.54.2]{nisthandbook}
\be\label{pwintegrals2}
I[P_n](t) = \int_{-1}^1 e^{i t x} P_n(x) dx = \frac{2}{(-i)^n}j_n(t),
\ee
where $j_n$ is the spherical Bessel function of order $n$. $I[T_n]$ can be evaluated
by expanding $T_n$ into a Legendre polynomial series and then using \eqref{pwintegrals2}.

Suppose now that we are 
given a box $B$ on which we have a continuous density
$\sigma(\y)$. For simplicity, we assume it is
approximated by the multivariate Legendre polynomial
\be
\label{papprox}
\sigma(\y)\approx 
\sum_{\|\alpha\|_\infty\le k-1} c_\alpha P_\alpha(\y), 
\quad \y = (y_1,\ldots,y_d) \in B.
\ee
Here, we are using standard multi-index notation with 
$\alpha=(\alpha_1,\ldots,\alpha_d)$,
$\|\alpha\|_\infty=\max_{i=1}^d |\alpha_i|$,
$c_\alpha=c_{\alpha_{1},\alpha_{2},\ldots \alpha_{d}}$, and 
$c_\alpha P_\alpha(\y)=c_{\alpha} \prod_{i=1}^d  
P_{\alpha_i}(y_i)$, where $P_n$ is 
the Legendre polynomial of degree $n$. 

\begin{lemma} \label{polycost}
The {\em cost} of computing $c_\alpha$ from $\sigma(\y)$ on a 
tensor product grid with $k^d$ points is $O(d \, k^{d+1})$. 
\end{lemma}

\begin{proof}
In one dimension, this follows from orthogonality. 
Using the Legendre basis (scaled to the box dimension), we have
\begin{equation} 
\label{val2poleq}
c_n  =  \frac{2n+1}{L} 
\int_{-L/2}^{L/2} \sigma(y) P_n \left(\frac{2y}{L} \right) dy
\approx \frac{2n+1}{2} \sum_{j=1}^k w_j \, \sigma(L\eta_j/2) \, P_n(\eta_j) 
\end{equation}
where $L$ is the side length of a box, and
$(\eta_j,w_j), j = 1,\dots,k$ are the nodes and weights for 
Gauss-Legendre quadrature on $[-1,1]$. 
In higher dimensions, the result follows from
the observation that one can compute this transform sequentially
in each dimension, assuming $\sigma$ is tabulated on a tensor product grid.
\end{proof}

\begin{definition} \label{val2poldef}
Let ${\cal T}_{val\to pol} \in \mathbb{R}^{k\times k}$ 
denote the one dimensional matrix from the preceding lemma which maps values of a function 
to its corresponding Legendre coefficients with entries
\[ {\cal T}_{val\to pol}(j,n) =  
\frac{2n+1}{2} w_n  \, P_{j-1}(\eta_n) \, \quad j,n=1,2,\ldots k\, .
\]
\end{definition}

The following lemma 
will be used when we wish to directly compute the Gauss transform
on a single box at a collection of target points.

\begin{lemma} \label{directbox}
Let $B$ be a box of side length $L$ (assumed to be 
centered at the origin for simplicity of notation).
Let $T$ be a box in its
neighbor list (see \cref{treedefs} and \cref{fignearfield}).
Then, for a source density on $B$ of the form
\eqref{papprox}, the potential at $\x = (x_1,\dots,x_d)$ is given by
\begin{align}
u(\x) &= \int_B e^{\|\x-\y\|^2/\delta} \sigma(\y)d\y  \nonumber \\
&= 
\sum_{\|\alpha\|_\infty\le k-1}  
c_\alpha \int_B e^{\|\x-\y\|^2/\delta} P_\alpha(\y)  d\y  \nonumber \\
&= 
\left( \frac{L}{2} \right)^d
\sum_{\|\alpha\|_\infty\le k-1}  
c_\alpha J_{\lambda}[P_{\alpha_1}](2x_1/L) \cdots
J_{\lambda}[P_{\alpha_d}](2x_d/L) 
\label{gdireval}
\end{align}
where
$\lambda = 2\delta/L$ and
$J_{\lambda}[P_{n}](x)$ is given by  
\eqref{gsintegrals}.
\end{lemma}
\begin{proof}
The formula follows immediately from \eqref{gsintegrals}, 
\eqref{papprox}, and the change of variables $y_{i} \to 2y_i/L$
to rescale the box $B$ to one of dimension $[-1,1]^d$.
\end{proof}

\begin{remark} \label{dircost}
Assuming the output points $\{ \x_i, i=1,\dots,k^d \}$ 
lie on a tensor-product
grid in $d$ dimensions, separation of variables can be used to 
accelerate the evaluation of \eqref{gdireval} at all targets.
It is easy to check that the cost of evaluation 
is $O(dk^{d+1})$ 
floating point operations.
\end{remark}

It is also straightforward to expand the Gauss transform over a box
with a polynomial density in the plane wave basis (assuming that the
interaction is within the cutoff range so that the plane wave
approximation is accurate).

\begin{lemma} \label{pwbox}
Let $B$ be a box of side length $L$ centered at the origin
and let $T$ be a target box where the potential induced by the
source density $\sigma$ is to be evaluated.
Then
\begin{align*}
u(\x) &= \int_B e^{\|\x-\y\|^2/\delta} \sigma(\y)d\y \\
&= 
\sum_{\|\alpha\|_\infty\le k-1}  
c_\alpha \int_B e^{\|\x-\y\|^2/\delta} P_\alpha(\y)  d\y \\
&= 
\sum_{m_1=-n_{f}/2}^{n_{f}/2-1} 
\dots \sum_{m_d=-n_{f}/2}^{n_{f}/2-1} 
e^{-i h \m \cdot \x/\sqrt{\delta}} v(\m)
\end{align*}
where $\m = (m_1,\dots,m_d)$
and
\begin{align*}
v(\m) &=
\left( \frac{L}{2} \right)^d e^{- \|\m\|^2 h^2/4}
 \sum_{\|\alpha\|_\infty\le k-1}  
c_\alpha I[P_{\alpha_1}]\left(\frac{Lm_1 h}{2\sqrt{\delta}} \right) \cdots
I[P_{\alpha_d}]\left( \frac{Lm_2 h}{2 \sqrt{\delta}} \right) 
\end{align*}
where
$I[P_{n}](x)$ is given by  
\eqref{pwintegrals} and $h$ is given in
\eqref{pwterms}.
\end{lemma}
\begin{proof}
The formula follows from 
\eqref{papprox}, \eqref{pwintegrals}, 
\cref{thm:planewave},
and a change of variables mapping $B$ to the box $[-1,1]^d$.
\end{proof}

\begin{remark} \label{pwcost}
\cref{pwbox} will be used to accelerate the calculation of
Gaussian interactions at fine levels in the FGT hierarchy.
It is easy to verify that, using separation of variables, 
the cost is $O(n_fk)$ in one dimension,
$O(n_fk^2 + n_f^2k)$ in two dimensions, and
$O(n_fk^3 + n_f^2k^2 + n_f^3k)$ in three dimensions.
For the values of $n_f$ and $k$ that are of practical interest,
this is more efficient than using the NUFFT.
\end{remark}

Having formed a plane wave expansion, it can be evaluated
efficiently on a tensor product grid (within the relevant range).

\begin{lemma} \label{thm:pweval}
Let $B$ be a box of side length $L$ centered at the origin.
Then
\[
u(\x) = 
\sum_{m_1=-n_{f}/2}^{n_{f}/2-1} 
\dots \sum_{m_d=-n_{f}/2}^{n_{f}/2-1} 
e^{-i h \m \cdot \x/\sqrt{\delta}} v(\m)
\]
can be evaluated on a tensor product grid using
$O(n_fk)$ operations in one dimension,
$O(n_fk^2 + n_f^2k)$ operations in two dimensions, and
$O(n_fk^3 + n_f^2k^2 + n_f^3k)$ operations in three dimensions.
\end{lemma}

\section{The new discrete FGT}  \label{sec:pfgt}

We turn now to the fast Gauss transform for point sources 
$\{ \y_j,j=1,\dots,N \}$ and targets 
$\{ \x_i,i=1,\dots,M \}$ in free space.
For a prescribed
precision $\varepsilon$, we first determine the 
``cutoff'' length $D_0 \sqrt{\delta}$ via \cref{D0}.
We then construct an adaptive level-restricted tree so that 
$D_0 \sqrt{\delta}$ 
is the exact dimension of a leaf node at some level of refinement.
For this, we first determine the median of the point distribution 
$\x_0 = (x_{0,1},\dots,x_{0,d}) \in\mathbb{R}^d$, with 
\be
x_{0,k} =  X_k^{max} - X_k^{min}
\ee
where 
\begin{align*}
X_k^{max} &= \max_{k=1}^{d}( \max_{i=1}^{M}  x_{i,k}, \max_{j=1}^N  y_{j,k} ), \\
X_k^{min} &= \min_{k=1}^{d}( \min_{i=1}^M  x_{i,k}, \min_{j=1}^N  y_{j,k} ).
\end{align*}
We set $\x_0$ to be the center of the ``root" node $B_0$.
The side length $L_0$ of the (square) root node is defined to be the 
smallest value
\be
L_0  =  2^{l_c} D_0 \sqrt{\delta}
\ee
which contains all the sources and targets, where
$l_c \in \mathbb{Z}^+$. We will refer to $l_c$ as the ``cutoff'' level.
We then build a {\it level-restricted} adaptive tree that sorts the 
sources and targets and refines boxes until each leaf node either has fewer than
$n_s$ sources and targets
or the tree has reached the cutoff level. Here $n_s$ is a user-supplied
constant. We use the standard language of tree data structures, with 
level $0$ defined to be the root node itself, and
level $l+1$ obtained recursively by subdividing each box at level
$l$ into $2^d$ equal parts. If $B$ is a box at level $l$,
the $2^d$ boxes at level $l+1$ obtained by its subdivision are 
referred to as its children.   

\begin{definition} \label{treedefs}
A tree is called {\em level-restricted} or {\em balanced} if two 
leaf nodes that share a boundary point are at most one level apart in the
hierarchy (\cref{pfgt_tree}).
In a level restricted tree, for any leaf node $B$ at level $l$, 
the leaf nodes sharing a
boundary point at the same level are called {\em colleagues} (including
$B$ itself). The leaf nodes at level $l+1$ sharing a
boundary point with $B$ are called {\em fine neighbors} and
the leaf nodes at level $l-1$ sharing a
boundary point with $B$ are called {\em coarse neighbors}. The union
of these three sets of boxes is called the {\em neighbor list}.
\end{definition}

The tree construction is carried out by the following method,
where by {\em points}, we mean the union of sources and targets.

\begin{algorithm}
{\em On input, we are given a collection of $N$ sources $\y_{j}$, $j=1,2,\ldots N$, 
a collection of $M$ targets $\x_{i}$, $i=1,2,\ldots M$, and $n_{s}$ the maximum number of points
in a box. The output is a level-restricted adaptive tree which has sorted the sources
and targets.}
  \caption{ Construction of level-restricted adaptive tree \label{alg1} }
  \begin{algorithmic}
   \State \multiline{$\sbtr$ Determine the center $\x_0$ of the root node $B_0$,
the box size $L_0$ and the cutoff length $D_0 \sqrt{\delta}$, as above.} 
\For{$l=0,1,\ldots l_{c}-1$}
\For{each nonempty box $B_{i}$ at level $l$}
\If{Number of sources or targets in $B_{i} > n_{s}$}
\State  $\sbtr$ Subdivide $B_{i}$ into $2^d$ child boxes at level $l+1$
\State $\sbtr$ Assign each point in $B_{i}$ to the child box in which is contained
\End
\End
\End
\State  \multiline{$\sbtr$ If the resulting adaptive tree is not level-restricted,
add additional refinement to enforce the condition using standard
algorithms (see, for example, \cite{sundar2008sisc}).
For the discrete FGT, also ensure that a leaf node at the cutoff
level containing more than $n_s$ points has no coarse neighbors.}
 \end{algorithmic}
\end{algorithm}

\begin{figure}[t]
\centering
\includegraphics[height=40mm]{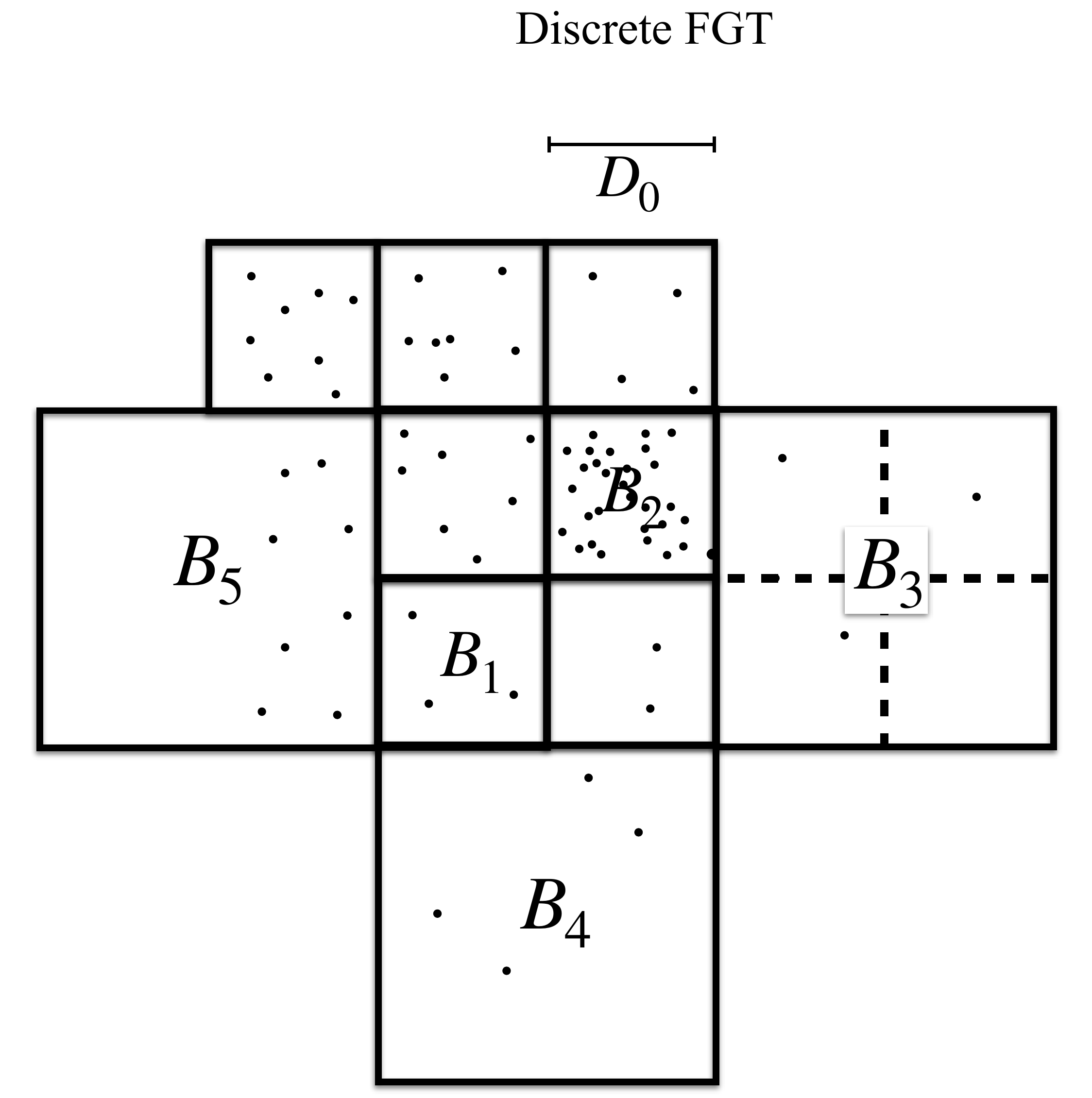}
\caption{\sf An adaptive tree for the discrete FGT.
The tree is used to resolve the point distribution down to the 
cutoff level, where the box dimension is $D_0 \sqrt{\delta}$. 
In this illustration,
$B_2$, a box at the cutoff level contains
more than $n_s$ points with a coarse neighbor $B_3$. 
This forces refinement, even though
$B_3$ had fewer than $n_s$ sources, in order to satisfy the 
criterion that the densely populated leaf node $B_2$ has no coarse
neighbors. Note that boxes $B_4,B_5$ do not require refinement. 
This will permit us to compute all non-trivial interactions for 
the points in $B_2$ using plane wave representations. 
}
\label{pfgt_tree}
\end{figure}

\begin{definition}
The {\em P-list} of a box $B$, denoted by $L_P(B)$, 
consists of all colleagues  
containing more than $n_s$ points. Note that the {\em P-list} is non-empty for
boxes $B$ at the cut-off level only.
The {\em D-List} of a leaf node $B$ at any level, denoted by $L_D(B)$, 
consists of all members of the neighbor list that are not in $L_P(B)$.
\end{definition}

In the new FGT, all processing is carried out at the cutoff level
(or above).
For a target point $\x$ in a box $B$, the corresponding potential $u(\x)$ can be
approximated by the sources in 
$L_D(B)$ and $L_P(B)$ with an error of $O(\varepsilon)$,
because of the exponential decay of the Gaussian, i.e., 
\be
u(\x)\approx u_D+u_P=\sum_{\y_j\in L_D(B)} G(\x-\y_j;\delta)q_j
  +\sum_{\y_j\in L_P(B)} G(\x-\y_j;\delta)q_j,
\label{pgt2}
\ee
For sources in $L_D(B)$, we compute the sum directly. 
It is straightforward to verify that, in a level-restricted tree, there 
are at most $C_d = 4^d - 2^d$ neighbors for any box 
(the maximum corresponding to 
a box at level $l$ surrounded by fine neighbors at level $l+1$).
Since each box in $L_D(B)$ contains at most $n_s$ sources, 
the cost of evaluating
the first sum in \eqref{pgt2} is bounded by a constant of the 
order $O(C_d n_s)$. 

For the second term $u_P$ in \eqref{pgt2},
we will approximate the Gaussian field using a 
plane-wave approximation, following \cref{thm:planewave}.
In $d$ dimensions, we use the tensor product version of the formula,
given by
\be\label{gpwapprox}
G(\x-\y;\delta) \approx \sum_{l=1}^{N_F}w_l e^{i \bk_l\cdot (\x-\y)}.
\ee
Here, $N_F= n_f(2D_0)^d$ with $n_f(2D_0)$ given in the 
third column of \cref{table1}  and 
$\bk_l, \, l = 1,\dots,N_F$ are the uniformly spaced points on a 
tensor product grid in Fourier space.
Recall, from the discussion in \cref{sec:spectral},
that the formula \eqref{gpwapprox} is accurate so long as 
$|x_i-y_i|\le 2 D_0 \sqrt{\delta}$ for
$i=1,\ldots,d$.
Letting $\bc_T$ be the center of box $T$ containing target $\x$,
and $\bc_S$ the center of box $S$ in $L_P(B)$, we define
the {\em outgoing} expansions $\phi_l(S), l = 1,\dots,N_F$
and the {\em incoming} expansions $\psi_l(B), l = 1,\dots,N_F$
by 
\be\label{multipoleexpansion}
\phi_l(S)=\sum_{\y_j\in S}e^{-i\bk_l\cdot (\y_j-\bc_S)} \, w_l \, q_j, 
\quad l=1,\ldots,N_F,
\ee
and
\be\label{diagonaltranslation}
\psi_l(T)=\sum_{S\in L_P(T)}e^{i\bk_l\cdot (\bc_T-\bc_S)}\phi_l(S).
\quad l=1,\ldots,N_F,
\ee
respectively.
It is straightforward to see that 
the contribution to $u_P$ at the target point $\x \in T$
is
\be\label{localexpansion}
u_{P}(\x)=\sum_{l=1}^{N_F}e^{i\bk_l\cdot (\x-\bc_T)}\psi_l(T). 
\ee

%
%
Thus, the NUFFT-based discrete FGT is described by the following five steps.


\begin{algorithm}
{\em On input, the set of $N$ sources and $M$ targets 
are assumed to be sorted on a level restricted adaptive tree, as described in
\cref{alg1}, for refinement parameter $n_s$. The output is the potential
evaluated at all target locations. 
}
  \caption{\label{alg2} NUFFT-based point FGT}
  \begin{algorithmic}
    \For{each box $B_{i}$ at level $l_{c}$} \Comment{Form outgoing expansion}
    \If{Number of sources in $B_{i} > n_{s}$}
    \State \multiline{$\sbtr$ Form the outgoing expansion $\phi_{\ell}(B_{i})$ from sources in $B_{i}$ 
    according to
\eqref{multipoleexpansion} using the type-1 NUFFT.}
 \End
\State $\sbtr$ Initialize the incoming expansion $\psi_l(B_{i})$ to zero.
\End

\For{each nonempty box $B_{i}$ at level $l_{c}$} \Comment{Gather expansions}
\State \multiline{$\sbtr$ Translate outgoing expansion for each box $S \in L_{P}(B_{i})$ to the center of $B_{i}$ 
and add to the incoming expansion $\psi_{l}(B_{i})$ according to~\eqref{diagonaltranslation}. (Each such translation is in diagonal form requiring only $N_F$ complex multiplications.)}
\End
\For{each nonempty leaf box $B_{i}$ at level $l_{c}$} \Comment{Evaluate incoming expansions}
\State $\sbtr$ Evaluate $\psi_l(B_{i})$ at
all targets in $B_{i}$ according to
\eqref{localexpansion}, using the type-2 NUFFT.
\End
\For{$l = 0, \ldots l_{c}$} \Comment{Direct interaction}
\For{each leaf box $B_{i}$ at level l}
\State \multiline{$\sbtr$ Compute the direct influence of the sources in $B_{i}$ at all targets in every box $B_{j}$ in $L_{D}(B_{i})$.}
\End
\End
  \end{algorithmic}
\end{algorithm}

Let $N_{box}$ denote the number of boxes $B_{i}$ which have greater than $n_{s}$ sources. 
The cost of forming outgoing expansions is 
$O(N_{box} N_F\log(N_F) + \log^d(1/\varepsilon) N)$,
the cost of gathering expansions is $O(3^d N_{box} N_F)$, and
the cost of evaluating expansions is 
$O(N_{box} N_F\log(N_F) + \log^d(1/\varepsilon) M)$.
Finally, the cost of the direct interactions is bounded by 
$O((4^d-2^d) n_s M)$.
Since $N_F$ is $O(\log^d(1/\varepsilon))$ by
~\cref{thm:planewave}, the total cost of the point FGT is 
\be\label{pfgtcost}
O\left(\log^d(1/\varepsilon)(M+N)\right).
\ee

For readers familiar with previous planewave-based FGT algorithms, it is 
worth noting that the expansion length $N_F$ is the same. Moreover,
the total number of translations per box is not much greater than
in the sweeping method of 
\cite{greengard1998nfgt,sampath2010pfgt,spivak2010sisc},
but permits a simple implementation on an adaptive tree.

\begin{remark}
It is also worth considering two extremes.
When $\delta \rightarrow 0$, the range of influence
of a Gaussian is vanishingly small. In that case, virtually
no expansion-related work is carried out, and the FGT basically
serves as an adaptive algorithm that assists in the determination of 
near neighbors for direct calculation.
When $\delta \rightarrow \infty$. on the other hand,
the Gaussian becomes smoother and smoother and the algorithm becomes trivial.
The tree construction terminates at the level of the root node, and
the FGT carries out one type-1 NUFFT and one type-2 NUFFT.
\end{remark}

\begin{remark}
For step 5 of the FGT (direct evaluation), a simple optimization
is to consider only the sources that lie in a ball
of radius $D_0\sqrt{\delta}$ centered at each target point rather than
all sources lying in the near neighbors.
In three dimensions, this reduces the cost by a factor of two or so.
\end{remark}

\section{The box FGT} \label{sec:vfgt}

We now consider the evaluation of the continuous or {\em box} Gauss 
transform~\eqref{bgt}, where the density $\sigma(\x)$ 
is assumed to be smooth and resolved on a level-restricted tree. 
We briefly describe how to construct such a tree given
a user-provided function that can evaluate $\sigma$ at any arbitrary point.
For this, let $k$ denote the desired polynomial order of approximation 
in each dimension and
let $\varepsilon$ be the desired precision.
Let $L_0$ denote the side length of the root box $B_0$ which we assume is
centered at the origin. 
Given an adaptive tree superimposed on $B_0$, 
we approximate the density
$\sigma(\x)$ on any box $B$ in this hierarchical data structure
as discussed in~\eqref{papprox}, i.e., 
\[
\sigma(\x)\approx 
\sum_{\|\alpha\|_\infty\le k-1} c_\alpha p_\alpha(\x), 
\quad \x = (x_1,\ldots,x_d) \in B.
\]
It is well-known that the 
approximation error in \eqref{papprox} 
can be estimated from the tails of the coefficients $c_\alpha$
outside the $l_2$ ball as
\be\label{funerr}
e=\sqrt{\frac{1}{N_{k,2}}\sum_{\|\alpha\|_{2} \ge k}|c_\alpha|^2},
\ee
where $N_{k,2}$ is the number of terms in the summation.
(See, for example, \cite{trefethen_approx,trefethen_mult}.) 
Using this error monitor, it is straightforward to  
recursively build an adaptive tree that resolves the density $\sigma$.

\begin{algorithm}
{\em Input is a tolerance $\varepsilon$, and the density function $\sigma$ which can be 
evaluated at an arbitrary point. We assume that a Legendre basis is used to represent the density on each box.
The output is a level-restricted
adaptive tree which resolves $\sigma$ to precision $\varepsilon$.}
  \caption{Resolving a continuous density on a 
level-restricted adaptive tree \label{alg3} }
  \begin{algorithmic}
    \For{level $l=0,1,\ldots$}  
    	\For{each box $B_i$ at level $l$}
    \State  $\sbtr$ Compute the density at the tensor product 
       Legendre nodes in $B_i$.
    \State  $\sbtr$ Calculate the expansion coefficient vector $c_{\alpha}$
       using \cref{val2poleq}, \cref{polycost}.
    \State $\sbtr$ Estimate the approximation error $e(B_i)$ via \eqref{funerr}.
     \If{$e(B_i)>\varepsilon \|\sigma\|_2$} 
    \State  $\sbtr$ Subdivide $B_i$ into $2^d$ child boxes at level $l+1$
    \End
    \End
    \End
    \State \multiline{$\sbtr$ If the resulting adaptive tree is not level-restricted,
add additional refinement to enforce the condition using standard
algorithms (see, for example, \cite{sundar2008sisc}).
The density on the refined boxes can be obtained either by interpolation
or by calling the user-provided function.
If we let $N_{\rm leaf}$ denote the number of leaf nodes, then there
are $N=N_{\rm leaf} \cdot k^d$ total points in the discretization.
It is easy to check that the cost of building the tree is approximately 
$O(dkN)$.}
\end{algorithmic}
\end{algorithm}

There are several important distinctions to highlight
between the continuous and discrete
FGTs: 
\begin{enumerate}
\item Unlike 
the discrete FGT, the tree construction does not terminate
at the cutoff level $l_c$. Deeper levels of refinement may be
needed to resolve the source density.
\item There are no empty boxes in the data structure.
In regions of space where the density is smooth, refinement will
terminate at a coarse level, but every leaf node is assumed to have 
a non-trivial density $\sigma(\x)$
tabulated on a tensor product grid with $k^d$ points.
\item
Most importantly, we can make very effective use of separation
of variables at every step of the algorithm, using the tools developed
in \cref{boxlemmas} and illustrated in \cref{fignearfield}, as we shall now
see.
\end{enumerate}

\begin{figure}[t]
\centering
\includegraphics[height=3in]{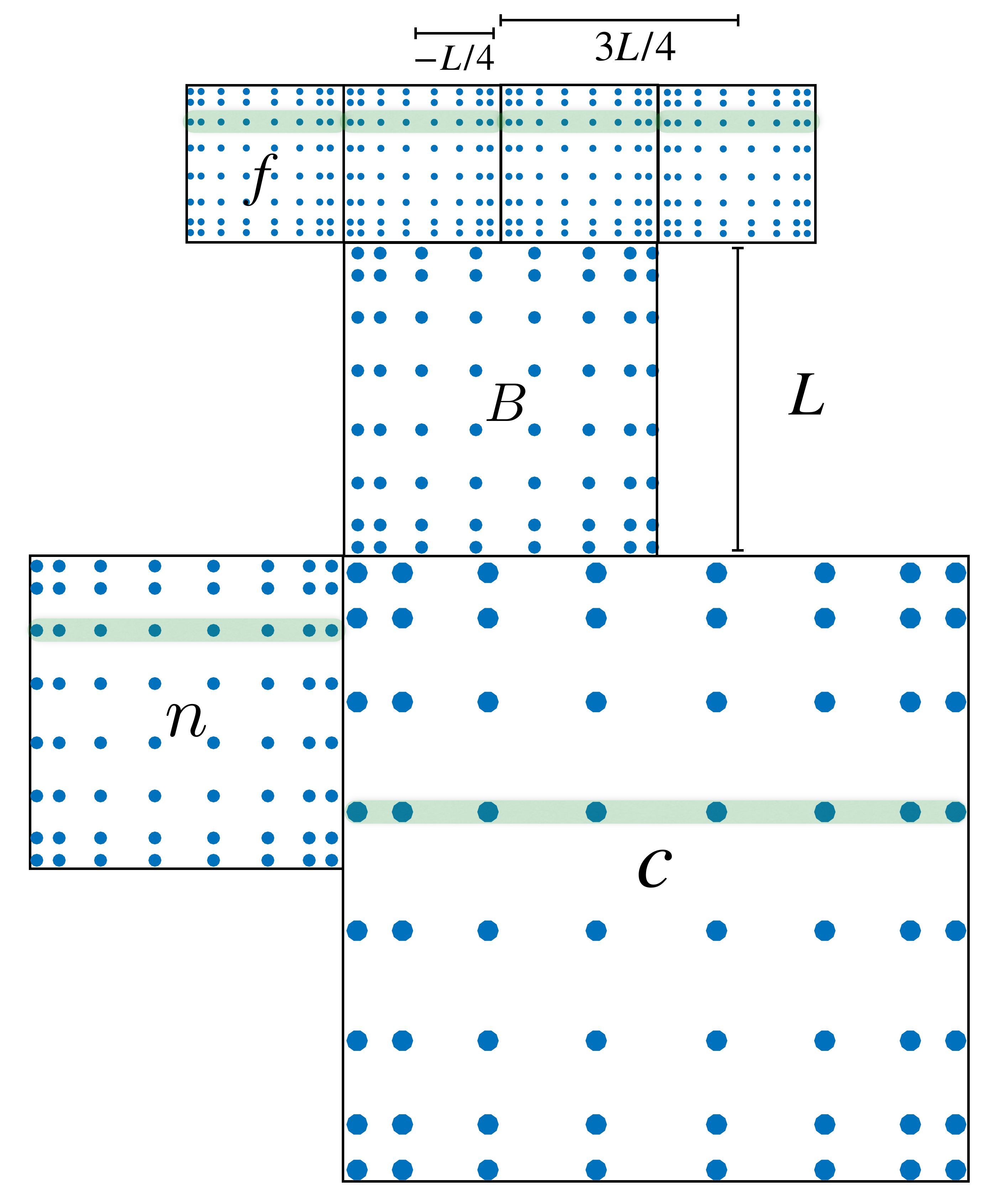} 
\caption{\sf 
For a leaf node $B$ at a coarse level,
its interactions with neighbors are computed directly.
There are only a finite number of possible locations for colleagues ($n$),
fine neighbors ($f$) and coarse neighbors ($c$). 
More precisely, it is straightforward to see that
in $d$ dimensions, there are at most $3^d$ colleagues,
$(2^d-1) \cdot 2^d$ fine neighbors and  
$(2^d-1) \cdot 2^d$ coarse neighbors. 
Because of the separability of the
Gaussian and the fact that the target grids are defined as tensor products,
the cost of direct calculation is $O(dk^{d+1})$ 
rather than $O(k^{2d})$, where $k^d$ is the number of grid
points in a box.
Note also that, in a level-restricted tree,
the tensor product grids on neighbors can only have a few possible
locations with respect to a leaf box $B$. 
Assuming $B$ is of side length $L$, 
the offset of the center of any fine neighbor in the $x_1$-direction 
has only four possible values:
$\{-3L/4,-L/4,L/4,3L/4 \}$.
}
\label{fignearfield}
\end{figure}

\subsection{Precomputation} \label{sec:precomp}

Before turning to the algorithm itself, we define several 
matrices that will be used throughout the box FGT and 
need to be computed once per level.
It is straightforward to see that, in a level-restricted tree,
the tensor product grids on neighbors can only have a few possible
locations with respect to a source leaf box $B$. 
As illustrated in  \cref{fignearfield},
assuming $B$ is of side length $L$, 
the offset of the center of any fine neighbor in the $x_1$
direction has only four possible values:
$\{-3L/4,-L/4,L/4,3L/4 \}$, which we denote by $s_1,s_2,s_3,s_4$,
respectively.

\begin{definition} \label{pol2potdef}
Letting
$\xi_1^{s_{l}},\dots,\xi_k^{s_{l}}$ be the $k$ Legendre nodes scaled to the 
fine neighbor length $L/2$, and offset by the relevant value $s_l,
l = 1\dots,4$,
we define by ${\cal T}^{s_l}_{pol\to pot} \in \mathbb{R}^{k\times k}$ 
the one dimensional matrix
${\cal T}^{s_l}_{pol\to pot} \in \mathbb{R}^{k\times k}$ 
with entries 
\[ {\cal T}^{s_l}_{pol\to pot}(i,j) = J_\lambda[P_{j-1}](2\xi_i^{s_{l}}/L), \]
where $\lambda = 2\delta/L$ for $i,j = 1,\dots,k$.
\end{definition}
Likewise, the center of a coarse neighbor also
has only four possible values: \\
$\{-3L/2,-L/2,L/2,3L/2 \}$, which we denote by $s_5,s_6,s_7,s_8$,
respectively.

\begin{definition} \label{pol2potdef2}
${\cal T}^{s_l}_{pol\to pot}$ for $l=5,\dots,8$ 
is defined as for fine neighbors but
with $\{ \xi_i \}$ located at scaled Legendre grid points
on a coarse neighbor
with offset $s_l$.
\end{definition}

The center of a colleague has only three possible values:
$\{-L,0,L \}$, which we denote by $s_9,s_{10},s_{11}$,
respectively.

\begin{definition} \label{pol2potdef3}
${\cal T}^{s_l}_{pol\to pot}$ for $s=9,10,11$ is defined as for other
neighbors but
with $\{ \xi_i^{s_{l}} \}$ located at a scaled Legendre grid on a colleague
with offset $s_l$. 
\end{definition}
The total cost at each level is negligible --- just
$11 k^2$ evaluations of the function $J_\lambda[P_{j-1}]$.

We will also make repeated use of the
one-dimensional transformation matrices $T_{pol\to pw}$ 
which map polynomial basis functions to their
plane-wave expansion coefficients via \eqref{pwintegrals}.

\begin{definition} \label{pol2pwdef}
We define ${\cal T}_{pol\to pw} \in \mathbb{C}^{n_{f}\times k}$ 
as the matrix with entries 
\[ {\cal T}_{pol\to pw}(i,j) = I[P_{j-1}](L i h/(2\sqrt{\delta})), \]
where $h$ is given in \eqref{pwterms}.
\end{definition}

Once the plane wave expansions are available for a box $B$,
we will make use of the precomputable matrices 
$T_{pw \to pot}$ for evaluation. 

\begin{definition} \label{pw2pot}
We define
${\cal T}_{pw\to pot} \in \mathbb{C}^{k\times n_{f}}$ 
as the matrix with entries 
\[ {\cal T}_{pw\to pot}(n,m) = e^{-i h m \xi_n/\sqrt{\delta}}, \]
where $h$ is given in \eqref{pwterms} and the
$\{ \xi_i \}$ are the Legendre nodes scaled to the box's side length $L$,
according to \eqref{thm:pweval}.
\end{definition}

Given a plane wave expansion for a box $B$ of the form
\[
u(\x) = 
\sum_{\m}
e^{-i h \m \cdot \x/\sqrt{\delta}} v(\m)
\]
we will also need to shift the expansion center to either
a parent, child or colleague.
We have already encountered the ``outgoing to incoming" shift 
in \eqref{diagonaltranslation}.
We define the corresponding map as
the diagonal matrix ${\cal T}_{out \to in} \in \mathbb{C}^{N_F \times N_F}$ 
with non-zero entries ${\cal T}_{out \to in}(l,l) = 
e^{i\bk_l\cdot (\bc_B-\bc_S)}$, where
$\bc_B$ is the center of $B$ and $\bc_S$
is the center of its colleague.
When shifting an outgoing expansion from child to parent, we will use
the diagonal matrix ${\cal T}_{c \to p} \in \mathbb{C}^{N_F \times N_F}$ 
with non-zero entries ${\cal T}_{c \to p}(l,l) = 
e^{-i\bk_l\cdot (\bc_S-\bc_B)}$, where 
$\bc_S$ is the center of one of $B$'s children.
For the reverse, shifting an incoming expansion from parent to child, 
we will use
the diagonal matrix ${\cal T}_{p \to c} \in \mathbb{C}^{N_F \times N_F}$ 
with non-zero entries ${\cal T}_{p \to c}(l,l) = 
e^{i\bk_l\cdot (\bc_B-\bc_S)}$, where 
$\bc_S$ is the center of $B$'s child.
Aspects of the box FGT are shown in \cref{bfgt_tree}.

\subsection{The full algorithm}

As noted above, the box FGT is a little more involved than the
discrete FGT, since the tree construction does not terminate
at the cutoff level $l_c$. 
Because of the separability 
of all steps in the algorithm, however, the performance in terms of
work per grid point is significantly better.
The algorithm combines 
hierarchical aspects of the FMM with the plane wave approach of
\cite{greengard1998nfgt,spivak2010sisc}. 

\begin{definition}
The P-List, $L_P(B)$, of a non-leaf box at level $l_c$ consists of all of 
its non-leaf neighbors at level $l_c$.
The D-List, $L_D(B)$, of a leaf box at level $l\le l_c$ consists of all of
its leaf neighbors at level $l\le l_c$.
\end{definition}

\begin{definition} \label{fflagdef}
For each box $B$, we also assign a flag $f(B)$ that indicates
whether it requires storage for a plane-wave expansion
We set $f(B)=1$ if it does: either
because it is at a level $l > l_c$ or because it is at level
$l_c$ and one of its colleagues is further subdivided. 
\end{definition}

\begin{algorithm}
{\em On input, the source density is assumed to
be provided on a level-restricted adaptive tree 
with $k$th order tensor product data on leaf nodes, 
as constructed in \cref{alg3}, with
cutoff level $l_c$ and maximum refinement level $l_{\rm max}$.
The output is the potentials evaluated at all tensor product grid points 
on all leaf boxes.}
  \caption{\label{alg4} The box FGT}
  \begin{algorithmic}
    \For{level $l=l_{\rm max}, l_c, -1$} \Comment{Form outgoing expansions}
    \State $\sbtr$ Compute the operators ${\cal T}_{pol\to pw}$ 
    and ${\cal T}_{val \to pw} = {\cal T}_{val \to pol} {\cal T}_{pol\to pw}$ 
    for level $l$ 
    \State (\cref{val2poldef,pol2pwdef}).
    \For{each leaf box $B_i$ with $f(B_i)=1$ at level $l$} 
    \State \multiline{$\sbtr$ Transform function values in $B_i$ to
    the plane-wave expansion using $T_{val\to pw}$.}
    \End
    \End
    \For{level $l=l_{\rm max}-1, l_c, -1$} \Comment{Merge outgoing expansions}
    \State $\sbtr$ Compute the operators ${\cal T}_{c \to p}$ for level $l+1$.
    \For{each non-leaf box $B_i$ at level $l$} 
    \State $\sbtr$ Translate and merge the plane-wave expansions of its children 
    using ${\cal T}_{c \to p}$.
    \End
    \End
    \State $\sbtr$ Compute the operators ${\cal T}_{out \to in}$ for level $l_c$.
    \For{each box $B_i$ at level $l_c$} \Comment{Gather expansions}
    \For{each box $S$ in $L_P(B_i)$}
    \State \multiline{$\sbtr$ Merge the outgoing expansion $\bphi(S)$ into the incoming
    expansion $\bpsi(B_i)$ using ${\cal T}_{out \to in}$.}
    \End
    \End
    \For{level $l=l_c,\ldots,l_{\rm max}-1$} \Comment{Push to children}
    \State $\sbtr$ Compute the operators ${\cal T}_{p \to c}$ for level $l$.
    \For{each non-leaf box $B_i$ with $f(B_i)=1$ at level $l$} 
    \State $\sbtr$ Translate the incoming expansion $\bpsi(B_i)$ to its 
    children using ${\cal T}_{p \to c}$. 
    \End
    \End
    \For{level $l=l_c,\ldots,l_{\rm max}$} \Comment{Evaluate expansion}
    \State $\sbtr$ Compute the operator ${\cal T}_{pw\to pot}$ for level $l$.
    \For{each leaf box $B_i$ at level $l$} 
    \State $\sbtr$ Evaluate the expansion $\bpsi(B_i)$ at the tensor product grid 
     using $T_{pw \to pot}$.
    \End
    \End
    \For{$l=0,\ldots,l_c$} \Comment{Direct interaction}
    \For{each source box $S$ at level $l$}
    \State \multiline{$\sbtr$ Compute the operators ${\cal T}_{pol\to pot}^{s_{j}}$, $j=1,2\ldots 11$, 
    and ${\cal T}_{val \to pot} = {\cal T}_{val \to pol} {\cal T}_{pol\to pot}$ 
    for level $l$.}
    \For{each target box $T$ in $L_D(S)$}
    \State \multiline{$\sbtr$ Compute the influence of $S$ at targets in $T$ using
     $T_{val \to pot}$ and add to 
   the contribution from the evaluation of local expansions.}
    \End
    \End
    \End
  \end{algorithmic}
\end{algorithm}

\begin{figure}[t]
\centering
\includegraphics[height=40mm]{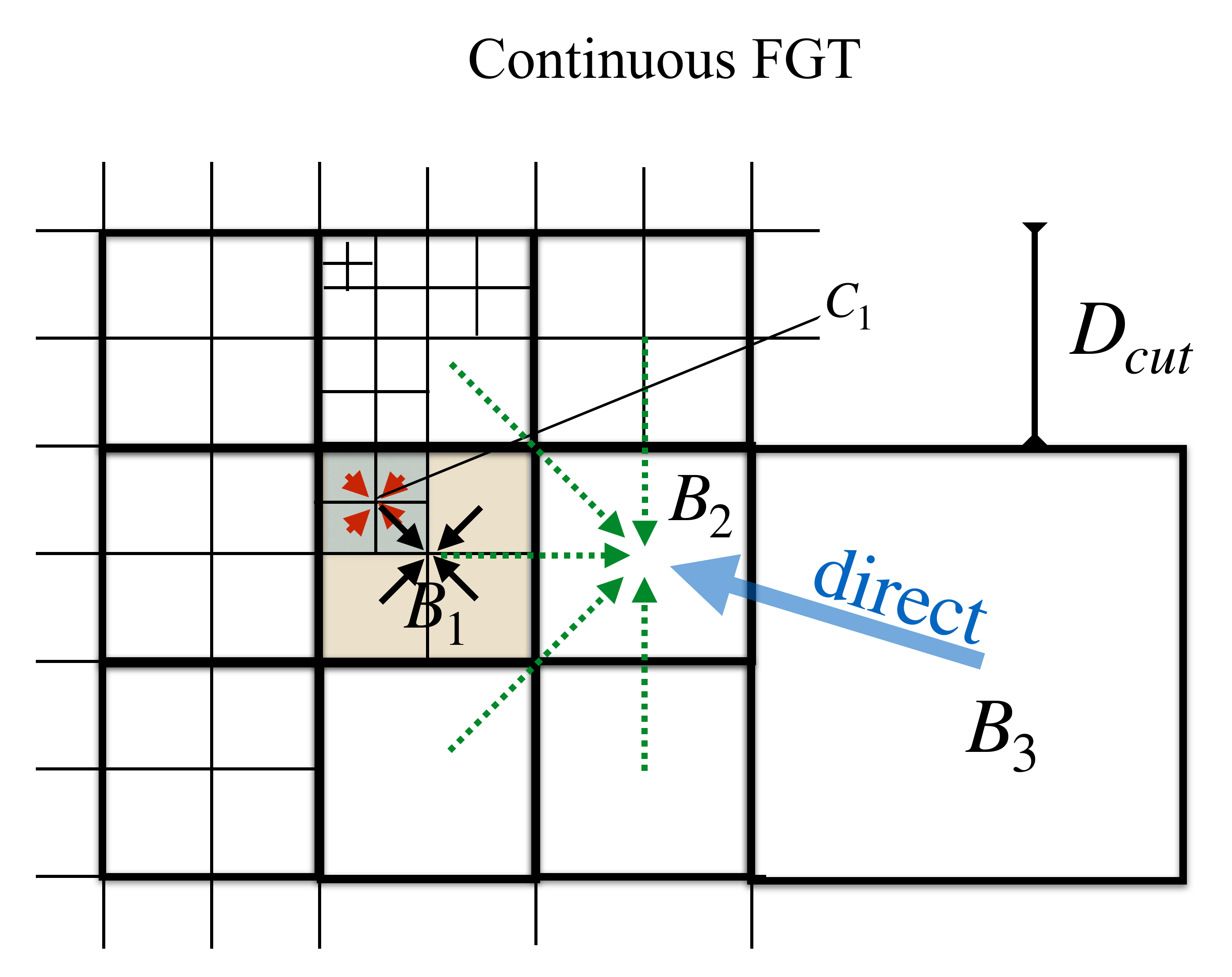}
\caption{\sf In the box FGT, outgoing plane wave expansions are 
recursively merged from fine levels up to the cutoff level (where
boxes are of side length $D_{cut}$), illustrated for box
$B_1$ and its refined child $C_1$. 
Outgoing expansions are first translated along the red arrows
to the center of $C_1$ from its children.
The black arrows indicate merger at the next level,
resulting in the outgoing expansion for $B_1$.
Once all outgoing expansions are available at the cutoff level,
they are translated to incoming expansions. This is illustrated
for box $B_2$. It acquires incoming expansions from its {\em colleagues}
(green dashed arrows) but the contribution from the source density in 
the coarse neighbor $B_3$
is obtained directly, using precomputed tables and separation of variables
(thick blue arrow).
Returning to a consideration of box $B_1$,
once it has acquired all relevant data as an
incoming expansion, that expansion is pushed recursively to its children
until leaf nodes are reached, at which point the plane wave expansion is
evaluated on the tensor product grid.
}
\label{bfgt_tree}
\end{figure}


A somewhat tedious calculation shows that the 
total computational cost is bounded by
\be
C \le c_1 N_{\rm leaf} n_f^d k + c_2 N_b n_f^d + c_3 3^d N_b n_f^d + c_4 3^d d k^{d+1} N_{\rm leaf},
\ee
where 
$N_{\rm leaf}$ is the number of leaf nodes,
$n_f$ is the number of plane wave modes in one dimension, and
$N_b$ is the total number of boxes in the adaptive tree.

Using the fact that $N_b \leq \frac{1}{1-2^{-d}} \, N_{\rm leaf}$, and that
the total number of
tensor product grid points is $N=N_{\rm leaf} k^d$,
we have
\be
C = O\left(\left[k(n_f/k)^d + 3^d(n_f/k)^d+3^d k\right] \, N\right).
\ee

\subsection{Resolving the output potential}  \label{sec4.2}

After completing the box FGT, we have accurately 
computed the potential at the given tensor product grid to the requested
precision.  In many applications, however,
one needs to evaluate the potential at auxiliary
target points. If the adaptive tree that resolves the input source density
were sufficient to resolve the output potential as well, 
then one could simply use polynomial interpolation for these points.
It is perhaps surprising that, {\em in an adaptive environment},
this is not always the case,
despite the fact that convolving with a Gaussian is a smoothing operator.
This phenomenon is best illustrated with a two-dimensional example.
In \cref{boxmove}, we show an adaptive tree resolving a source distribution
(top left) and the adaptive tree needed to resolve the convolution of that
distribution with a Gaussian. The difficulty here is that
a sharp feature can be diffused a short distance while remaining too sharp
to be well resolved by the original adaptive grid. 
In this example, the
input data consists of two sharply peaked Gaussians 
\[
\sigma_f(\x)= \sum_{i=1}^{2} e^{-\|\x-\x_i \|^2/\alpha_i}
\]
with $\alpha_1=\alpha_2=10^{-4}$. We let $\delta=4\times 10^{-3}$, and
set
the polynomial approximation order to $k=8$ and the requested 
precision to $\varepsilon=10^{-12}$.
We then evaluate the potential at four million equispaced tensor product points
on the unit box. A level-restricted adaptive tree is constructed
resolve the input data to the requested precision. The total number
of boxes is $5,709$ and among them $4,282$ are leaf boxes. If we use
the potential values at the tensor grid of the input tree to 
calculate the potentials
at four million equispaced points by interpolation. The relative $l_2$ error is
about $10^{-7}$. To remedy this loss of precision, we introduce refinement 
(and coarsening) capabilities below. After refinement,
the relative $l_2$ error is about $1.3\times 10^{-12}$. 
During the refinement stage,
$3,400$ boxes are added. During the coarsening stage, $3,824$ boxes are deleted.
Enforcing the level-restriction added only $16$ boxes and the total number of 
boxes in the new tree is $5,301$. 

\begin{figure}[t]
\centering
\includegraphics[height=40mm]{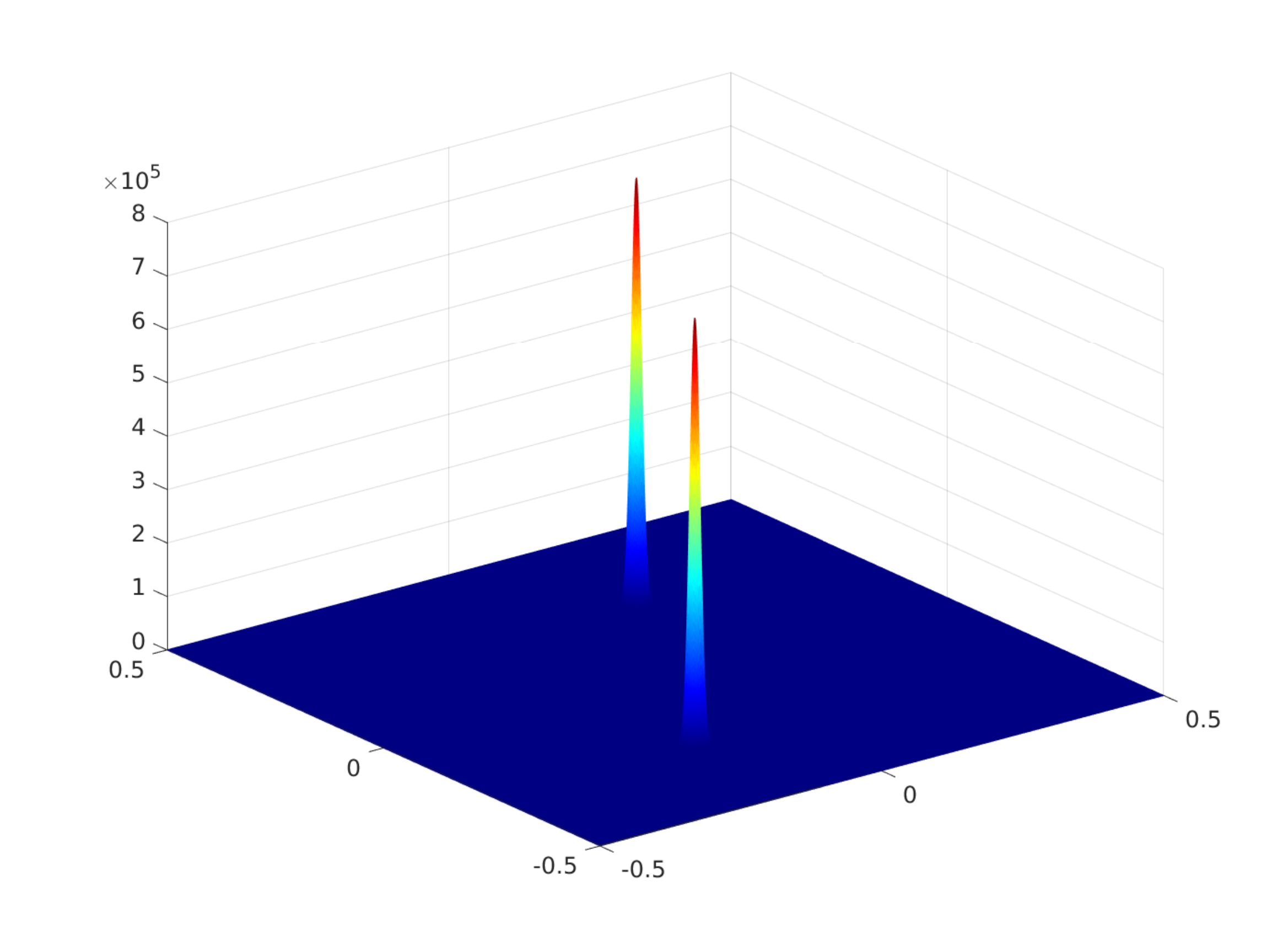}
\hspace{0.4in}
\includegraphics[height=40mm]{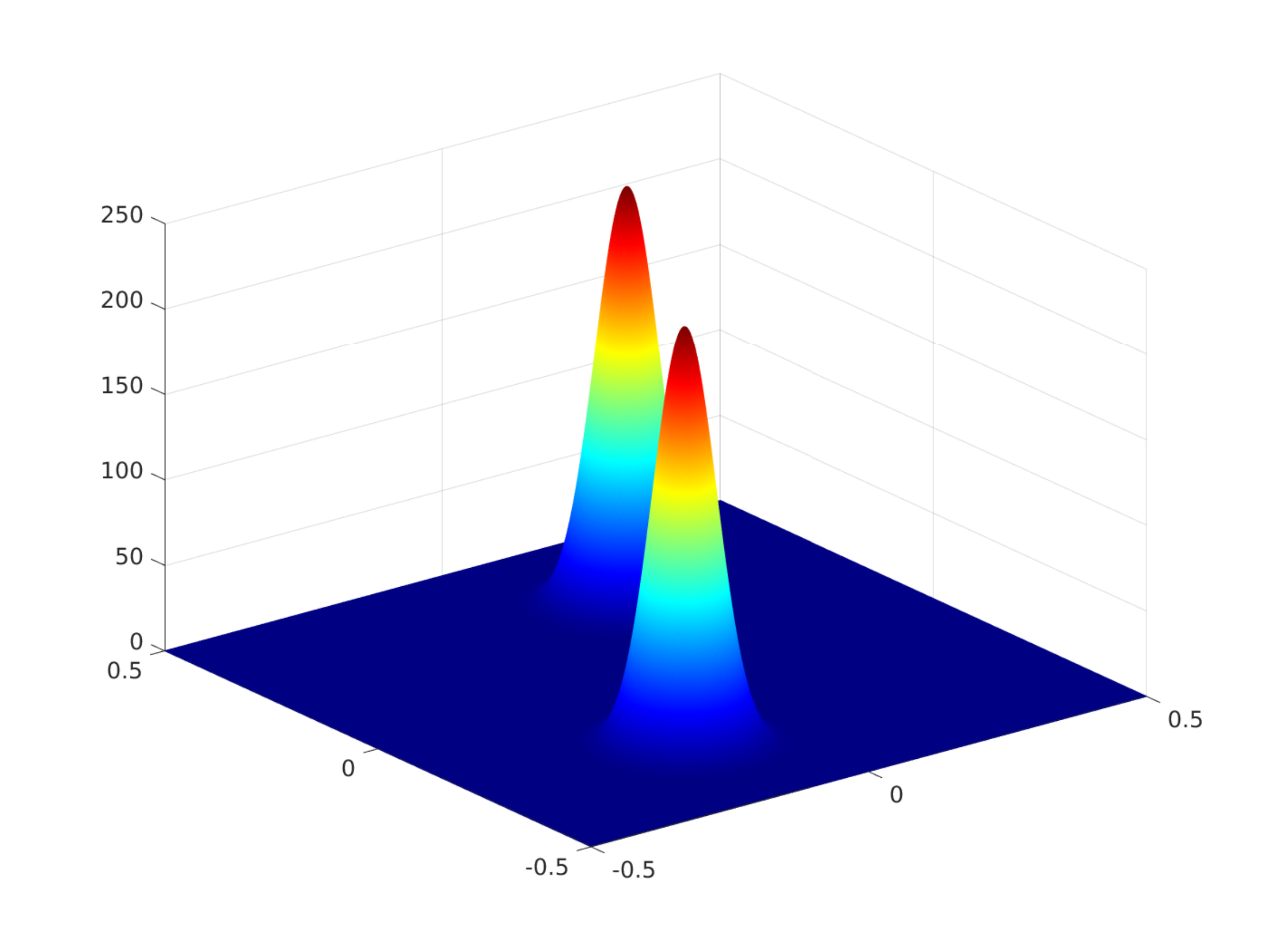}

\vspace{5mm}

\includegraphics[height=40mm]{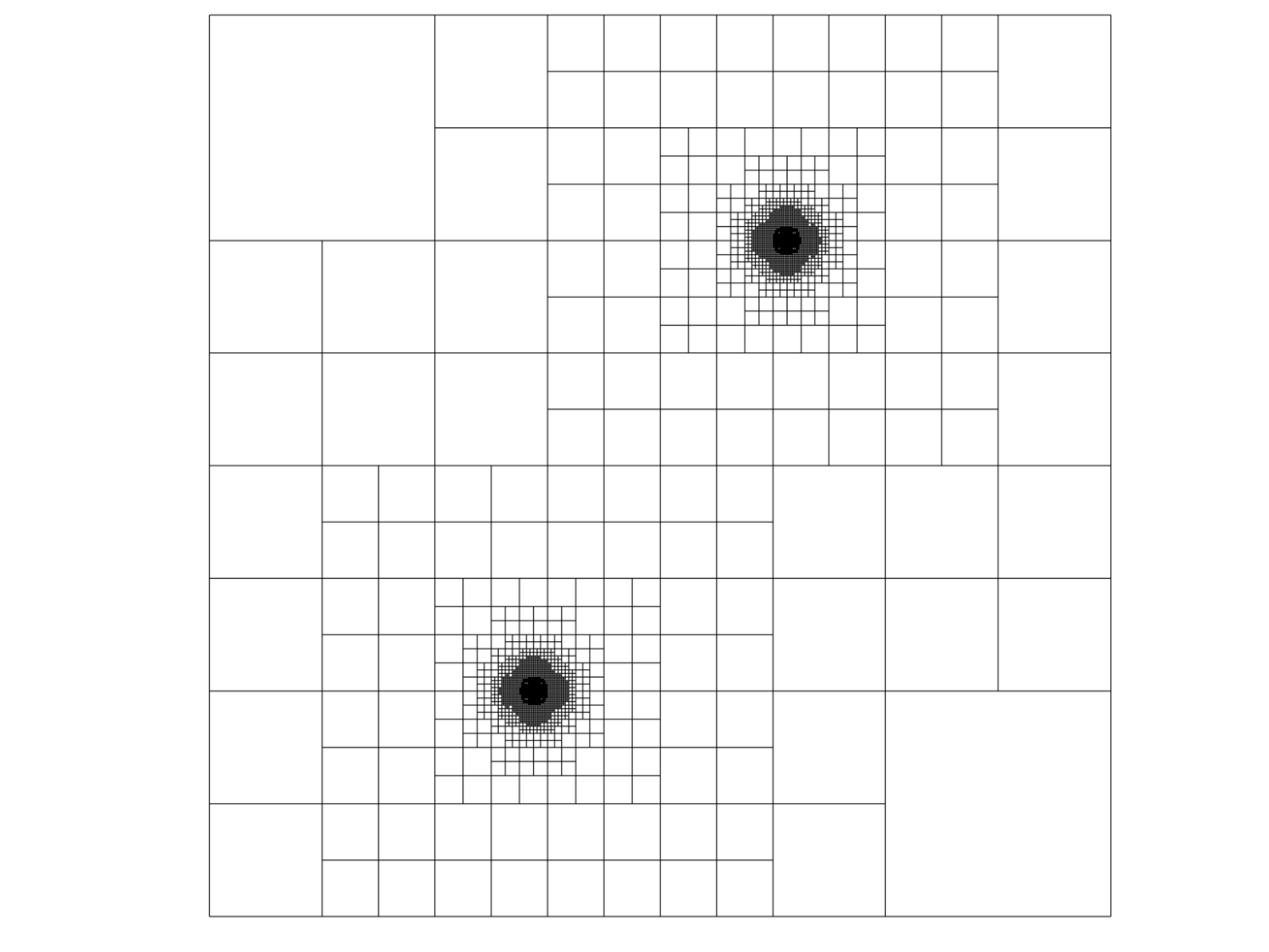}
\hspace{0.4in}
\includegraphics[height=40mm]{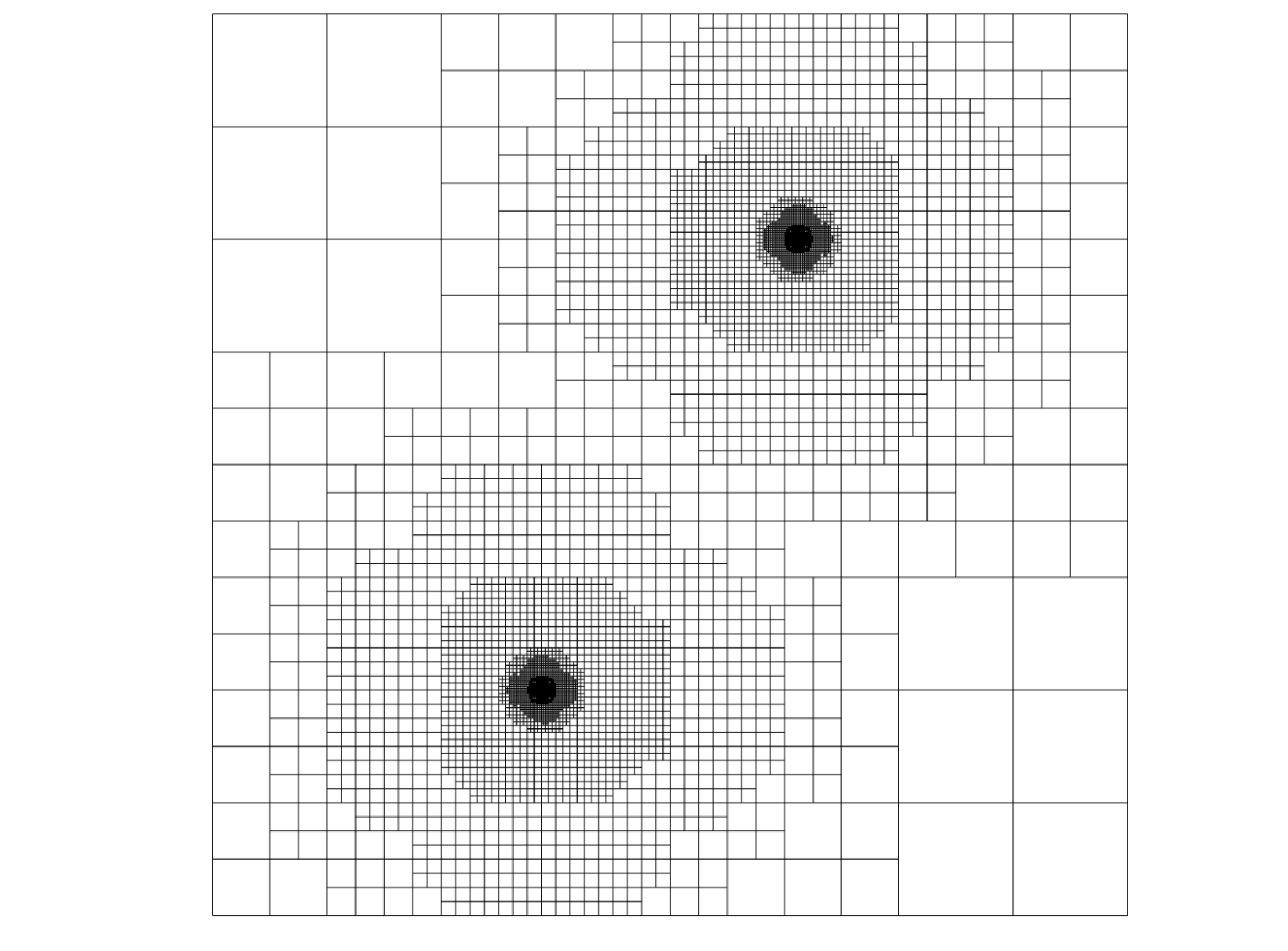}

\vspace{5mm}

\includegraphics[height=40mm]{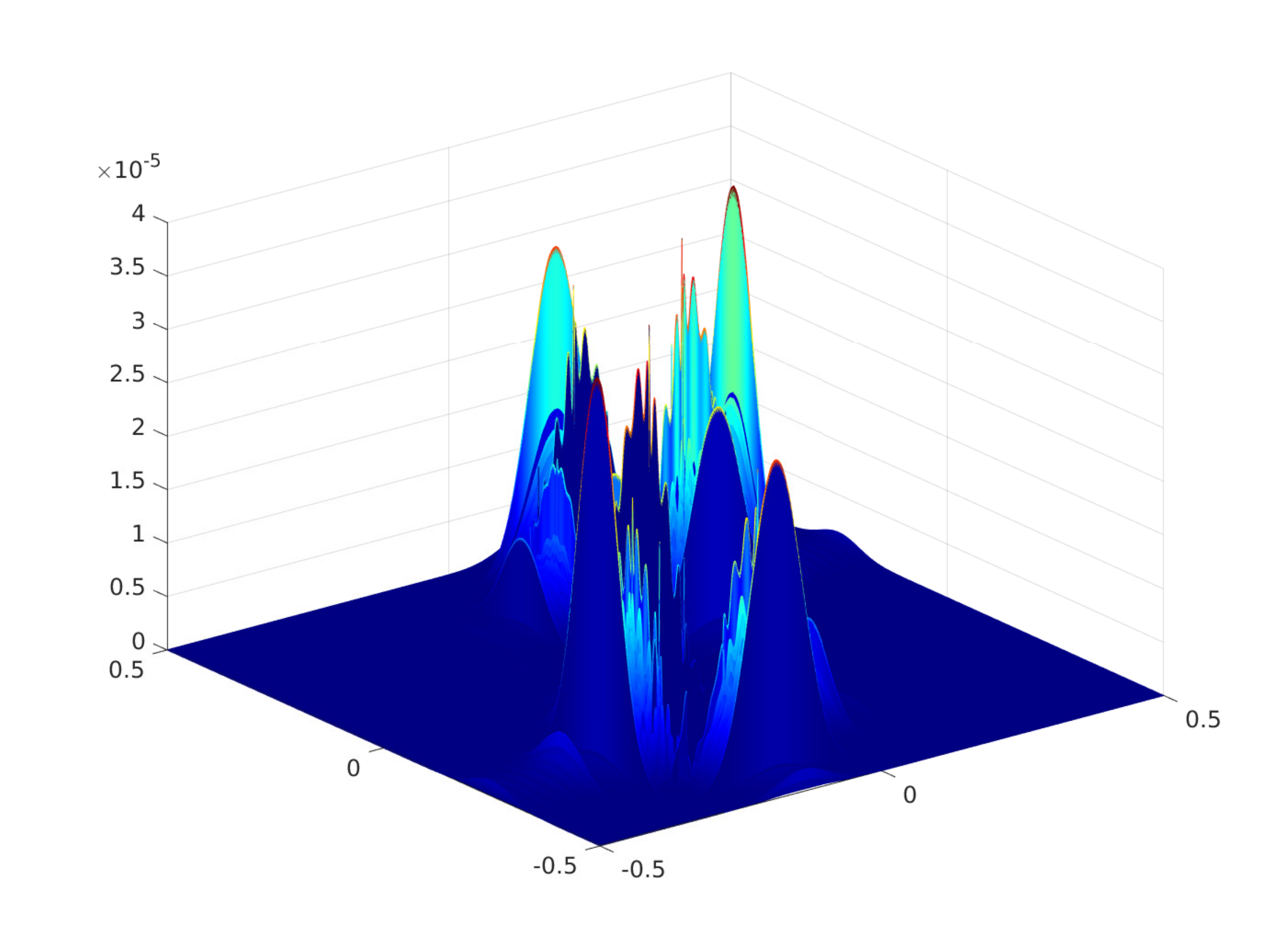}
\hspace{0.4in}
\includegraphics[height=40mm]{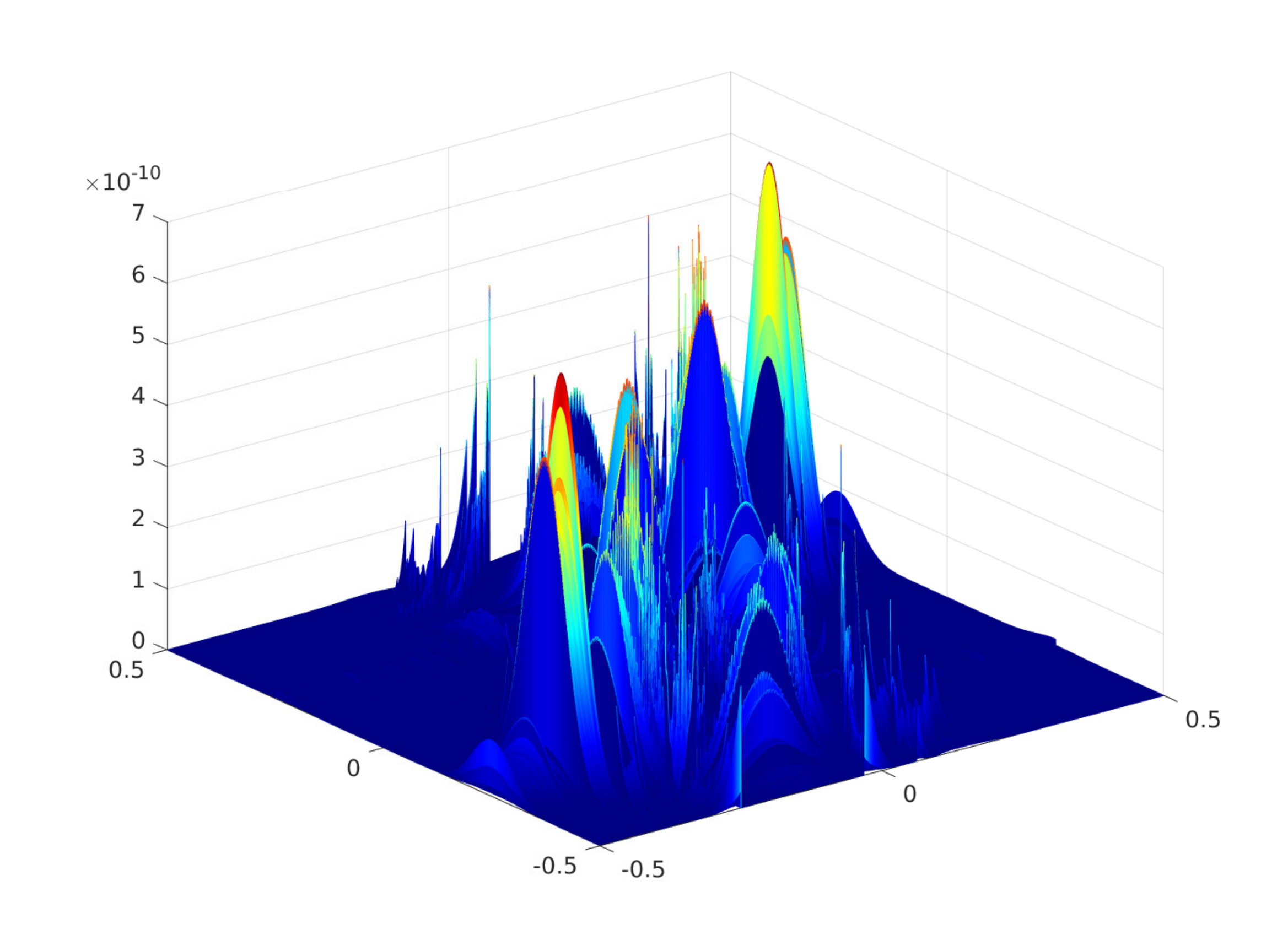}
\caption{\sf Illustration of the necessity of a new output 
  tree for potential evaluation at arbitrary targets by local 
    polynomial interpolation.
  Top left: the input data.
  Top right: the output potential (i.e., the Gauss transform of the input data).
  Middle left: the quadtree that resolves the input data to $12$ digits of accuracy.
  Middle right: the quadtree that resolves the output potential to $12$ digits
  of accuracy.
  Bottom left: the absolute errors of the potentials at four million equispaced
  target points using the input quadtree ($y$-axis in $10^{-5}$ scale);
  Bottom right: the absolute errors of the potentials at four million equispaced
  target points using the correct output quadtree ($y$-axis in $10^{-10}$ scale).}
\label{boxmove}
\end{figure}

The refinement algorithm, guaranteeing that the output is resolved to any requested precision,
is straightforward (\cref{alg5}).

\begin{algorithm}[t]
  \caption{Tree refinement}
{\em On input, we have an adaptive,
 level-restricted tree $T_0$ with $N_0$ boxes, 
that resolves the density $\sigma$ and for
which we have computed the potential $u$ on a tensor product grid
of order $k$ for each leaf node, using \cref{alg4}.
On output, we obtain a
new adaptive tree $T_1$ obtained by recursive subdivision ensuring that
the potentials on each leaf node are well approximated by tensor product 
polynomial approximation to the prescribed precision $\varepsilon$.}
\begin{algorithmic}
    \State \multiline{$\sbtr$ Create an integer array ${\sf ifrefine}$ of 
sufficient length $N_{\rm max} \gg N_0$ and initialize
it to zero.}
    \For{box $i=1,\ldots,N_0$}
    \If{$B_i$ is a leaf box}
    \State \multiline{$\sbtr$ Convert function values at the tensor product grid to the
           corresponding  polynomial expansion using \cref{polycost}.}
    \State $\sbtr$ Calculate the estimated approximation error $e(B_i)$ 
via \eqref{funerr}.
    \If{$e(B_i)>{\varepsilon}$}
    \State $\sbtr \, {\rm ifrefine}(B_i)=1$.
    \End
    \End
    \End
    \For{box $i=1,2,\ldots$} 
    \If{(${\rm ifrefine}(B_i)=1$)}
    \State $\sbtr$ Divide $B_i$ into $2^d$ child boxes and append them to the 
tail of the box list.
    \If{($f(B_i)=1$) from \cref{fflagdef}},
    \State $\sbtr$ Translate the plane-wave expansion $\bpsi(B_i)$ to its children.
    \End
    \State \multiline{$\sbtr$ Evaluate the translated expansion at the tensor product grid 
     on each child using $T_{\rm pw \to pot}$ and add direct contributions
     using $T_{pol\to pot}$ if present.}
    \End
    \End
    \State $\sbtr$ If desired, rebalance the tree so that it is level-restricted.
  \end{algorithmic}\label{alg5}
\end{algorithm}
\begin{remark}
 \cref{alg5} may require operators such as $T_{pol\to pot}$ for 
boxes which do not satisfy the level-restricted property. Since
the tables are of dimension $k \times k$, the additional cost is negligible.
Note that the maximum number
of refinements for a leaf box at level $l$ is $l_{\rm max}-l$.
That is, maximum tree refinement can never be {\em deeper} than the tree 
resolving the source data. 
\end{remark}

\begin{remark}
  The refinement algorithm requires precomputation of more local tables than those in
  Definitions \ref{pol2potdef}--\ref{pol2potdef3}. Let $N_{\rm refine}$ denote the maximum number of refinements. A careful calculation shows that the total number of local tables
  needed is bounded by $2^{N_{\rm refine}+2}\cdot (N_{\rm refine}+3)+1$. Furthermore, instead
  of splitting the local tables into three categories as in Definitions
  \ref{pol2potdef}--\ref{pol2potdef3}, one may place all local tables into one big array
  and retrieve the correct index of the table for a given pair of source and target boxes
  by calculating the ratio of the distance between their centers to
  $\tilde{L} = L/2^{N_{\rm refine}+2}$, with $L$ the side length of the source box.
\end{remark}

While \cref{boxmove} suggests that refinement can create a large number of new boxes,
the mollifying behavior of the Gaussian also permits tree
{\em coarsening} if the function is deemed to be over-resolved.
The algorithm for tree coarsening is also simple and presented in~\cref{alg6}.

\begin{algorithm}[t]
{\em On input, we are given an adaptive tree $T_0$ on which the potential is 
resolved to the prescribed precision, but may be over-resolved.
On output, we have a coarsened tree $T_1$ that has fewer boxes but still 
resolves the input function.}
  \caption{\label{alg6} Tree coarsening}
  \begin{algorithmic}
    \For{level $l=l_{\rm max}-1, 0, -1$} 
    \For{each non-leaf box $B_i$ at level $l$}
    \If{all child boxes of $B_i$ are leaf boxes} 
   \State \multiline{$\sbtr$ Convert function values at the tensor product grid to the
           corresponding polynomial expansion using \cref{polycost}.}
    \State \multiline{$\sbtr$ Evaluate the potential at the $k^d$ tensor product grid points
          on $B_i$ based on which child the grid point is in.}
   \State \multiline{$\sbtr$ Convert function values at the grid points on $B_i$ to the
           corresponding polynomial expansion using \cref{polycost}.}
    \State $\sbtr$ Compute the error estimate $e(B_i)$ using \eqref{funerr}.
    \State $\sbtr$ If $e(B_i)\le \varepsilon$, delete the children of $B_i$ from the 
    tree and set $B_i$ to be a \\ \hspace{0.6in} leaf node.
    \End
    \End
    \End
    \State $\sbtr$ If desired, rebalance the tree so that it is level-restricted.
  \end{algorithmic}
\end{algorithm}

\section{Periodic boundary conditions} \label{sec:periodic}
Periodic boundary conditions are easily implemented as follows. 
When $l_c\le 1$, i.e., when $\delta$
is sufficiently large, the plane-wave expansion of the periodic Gaussian 
$G_p$ on the entire box $B$ (the root node) requires a modest
number of terms and can be used directly.
When $l_c\ge 2$, 
the influence of distant image boxes at the relevant scale is exponentially
small and one can simply modify the definition of the
{\em colleague} list by include one layer of image boxes around 
$B$ at the cutoff level. Lists $L_D$ and $L_P$ also require slight 
modification to take the additional neighbors into account.

\section{Numerical results} \label{sec:numer}
In this section, we illustrate the performance of the new FGTs
for discrete and continuous sources. 
The code is written in Fortran and compiled using gfortran 9.4.0 with 
optimization flag -O3.
All experiments were run in single-threaded mode on a 2.40GHz
Intel(R) i9-10885H CPU. 
We will present numerical results in two and three dimensions, 
because the one-dimensional results don't reveal any additional interesting
features. 

\subsection{Point FGT experiments}

From the cost breakdown discussed just before the estimate
\eqref{pfgtcost},
it is easy to see that the performance of the point FGT depends 
on $\delta$, which controls the tree depth and therefore
the total number of boxes (and hence the average number of points per 
leaf box).
This dependence can be largely attributed to outgoing to incoming plane wave translations (gather expansions step in~\cref{alg2})
which requires $O(3^d N_{box} N_F)$ work.
This is true for virtually all tree-based algorithms, including 
FGTs and FMMs. An exception is 
the one-dimensional FGT based on the sum-of-exponential approximation 
in \cite{jiang2022cicp}, which can proceed without a tree data structure,
using the ordering of points on the line instead.

As a result, one can get excessively optimistic performance estimates
by fixing $\delta$ and increasing $N$ alone.
A more telling experiment is to reduce $\delta$ as we increase $N$, which 
we do following the approach used 
in~\cite{jiang2022cicp} for two dimensions and
in~\cite{spivak2010sisc} for three dimensions, where the points
are placed on a curve or surface, respectively so that adaptive refinement
is necessary.
In two dimensions, we
fix $N \delta = 400$ and let $N$ vary from $5,000$ to
$20,000,000$.
In three dimensions, we use the pairs
\cite{spivak2010sisc} 
\begin{align*} 
&(\delta=0.01,N=10^6),
(\delta=0.007,N=3.4 \times 10^6), 
(\delta=0.005,N=8 \times 10^6), \\
&(\delta=0.004,N=15.6 \times  10^6),
(\delta=0.003,N=27 \times 10^6).
\end{align*}
The corresponding
timings are presented in \cref{pfgtN}. 

\begin{figure}[th]
\centering
\includegraphics[height=40mm]{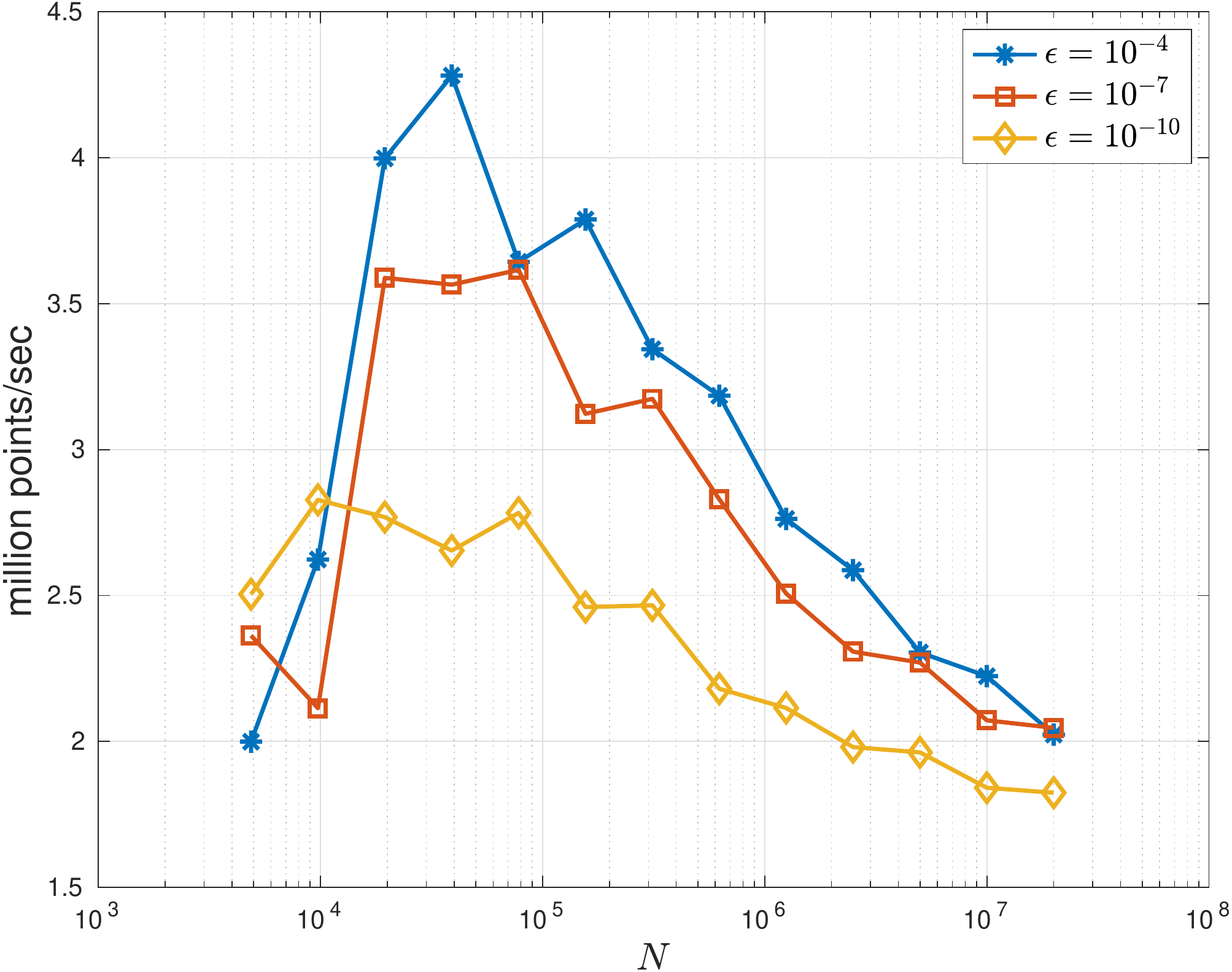}
\hspace{0.4in}
\includegraphics[height=40mm]{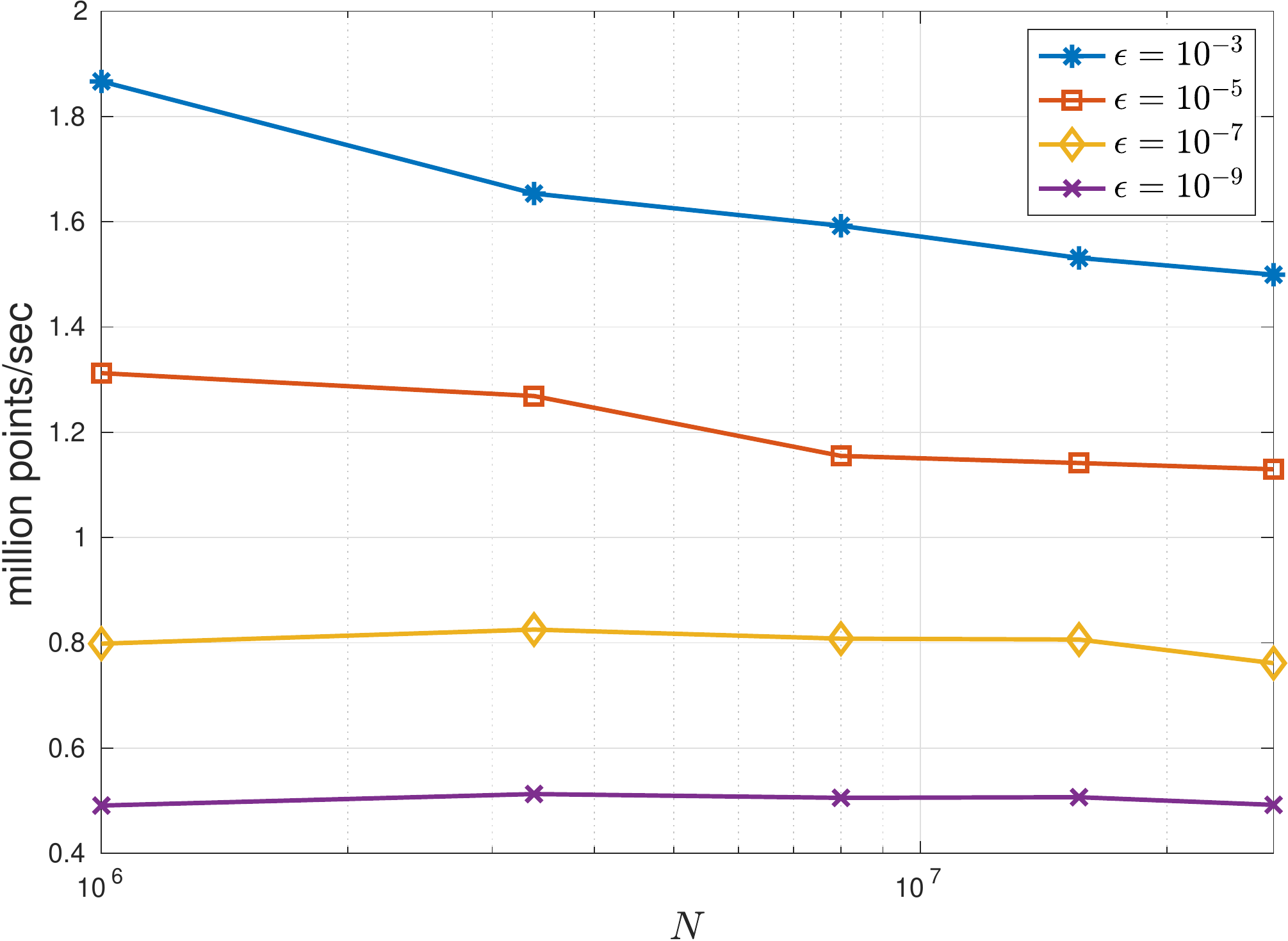}
\caption{\sf Linear scaling of the point FGT. 
$N$ here is the total number of points and the 
$y$-axis indicates the ``throughput" of the FGT measured
  in units of million points/second. 
(Left) a 2D problem set up to correspond to Fig. 11 in \cite{jiang2022cicp},
with points on a curve.
(Right) a 3D problem setup to match Tables 4.1-4.4 in \cite{spivak2010sisc},
with points distributed on a surface.
}
\label{pfgtN}
\end{figure}

A second informative experiment is to 
fix the number of source points and vary $\delta$. For this, we let
$N=10^7$ and let $\delta$ range from $10^{-10}$ to $10^{-1}$.
We set the target points to be the sources themselves. 
(Evaluating the field at additional
target points comes at very little cost.) 
In both the two-dimensional and three-dimensional case,
we distribute the sources either uniformly in the unit box
$B_d\in \mathbb{R}^d$  or uniformly on the surface of the perturbed 
sphere $S_{d-1}\in \mathbb{R}^d$ (which is highly nonuniform in the 
ambient space).
Timing results are shown in \cref{pfgt2d} and \cref{pfgt3d}.
\begin{figure}[th]
\centering
\includegraphics[height=40mm]{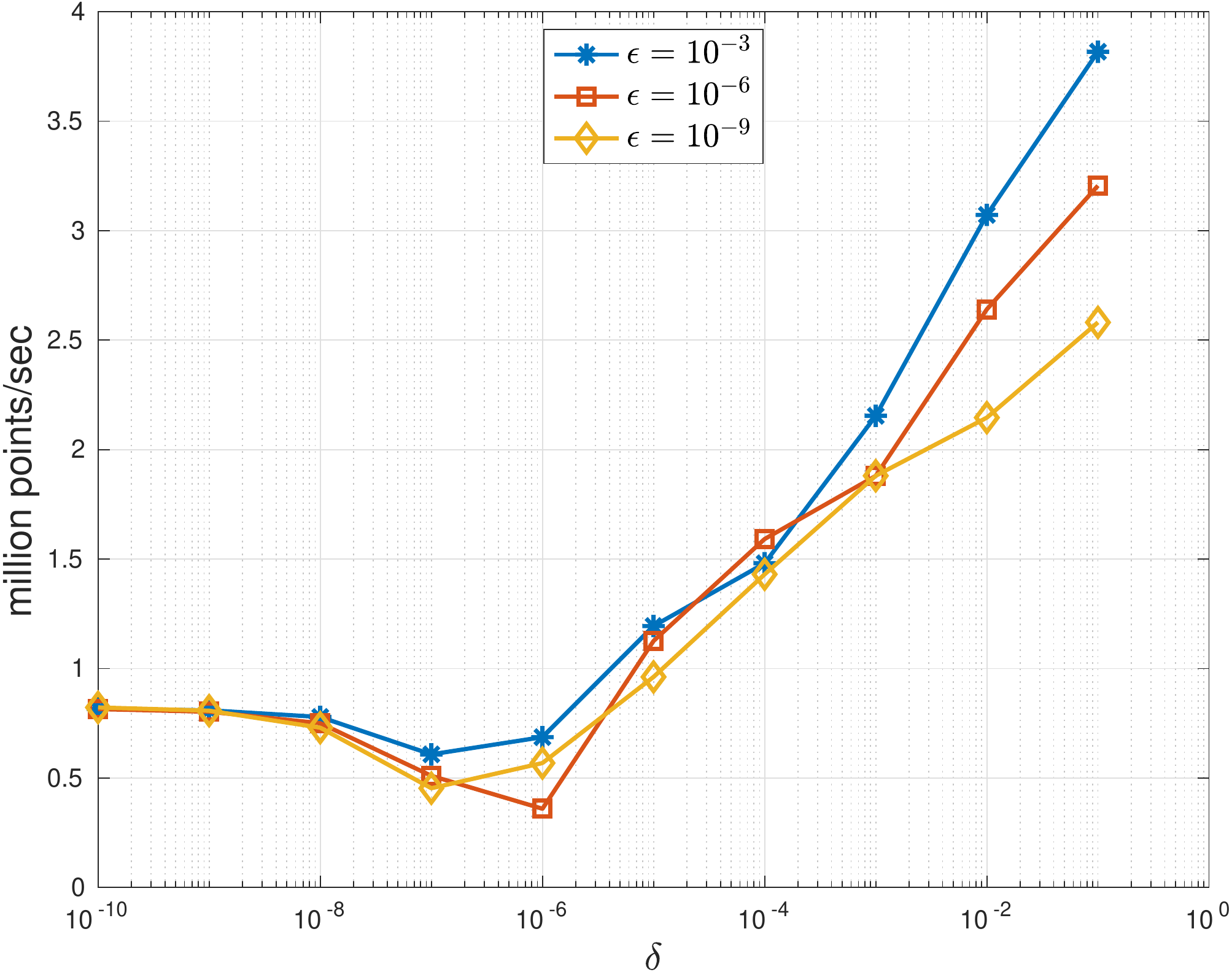}
\hspace{0.4in}
\includegraphics[height=40mm]{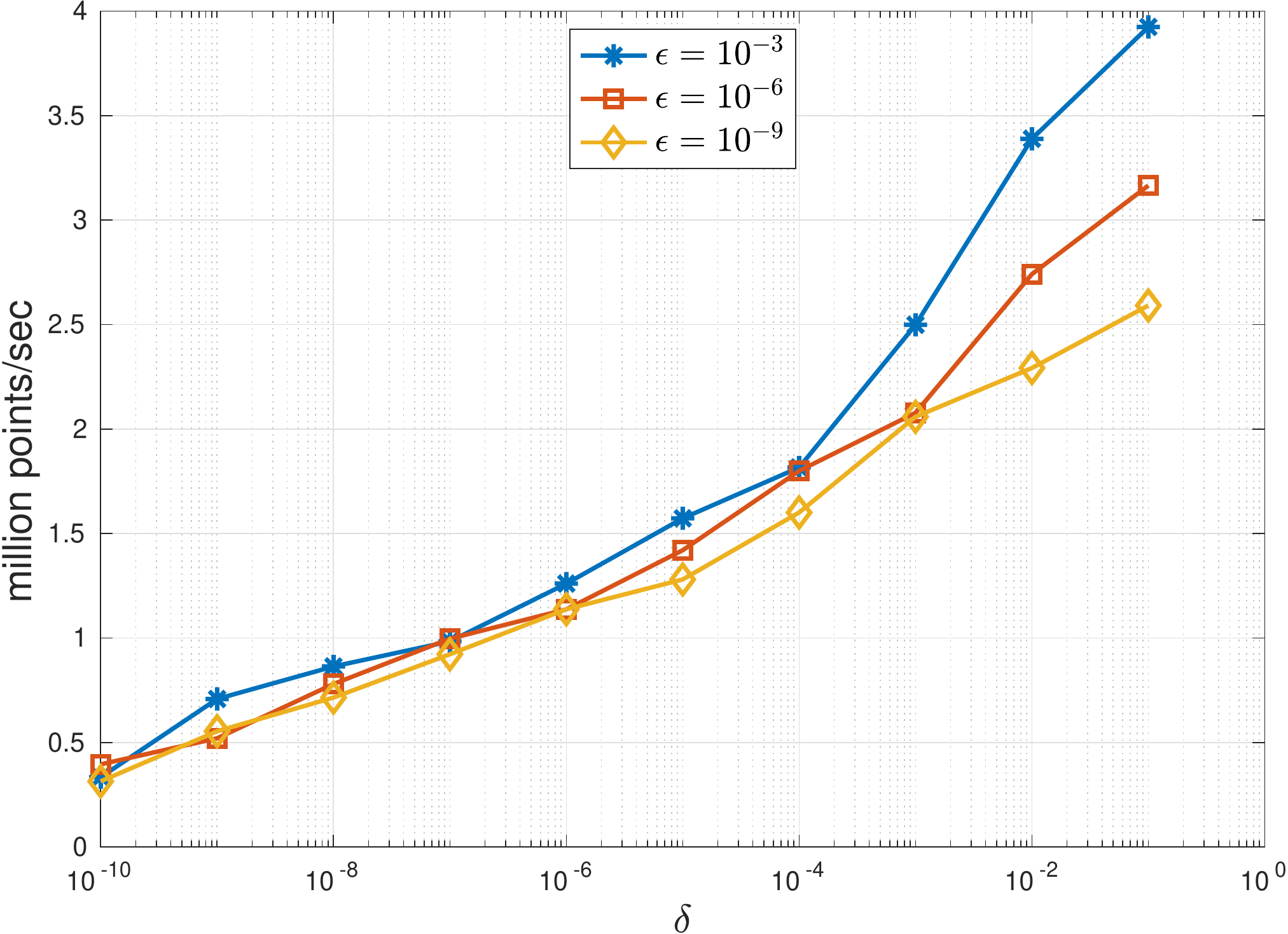}
\caption{\sf Performance of the 2D point FGT 
as a function of Gaussian variance, with throughput measured in units of
  million points per second:
(left) a uniform distribution, (right)
a highly nonuniform distribution.}
\label{pfgt2d}
\end{figure}
\begin{figure}[th]
\centering
\includegraphics[height=40mm]{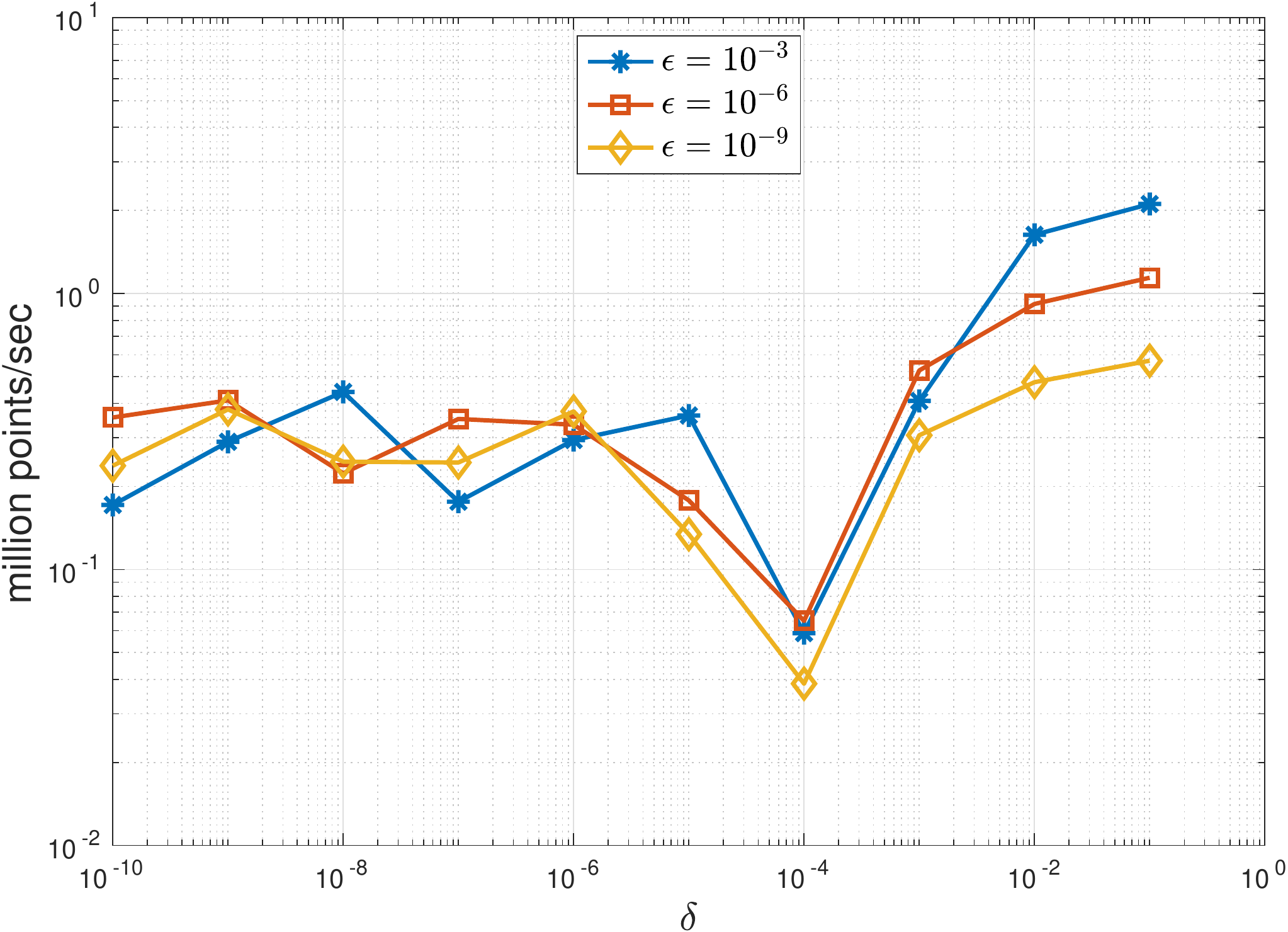}
\hspace{0.4in}
\includegraphics[height=40mm]{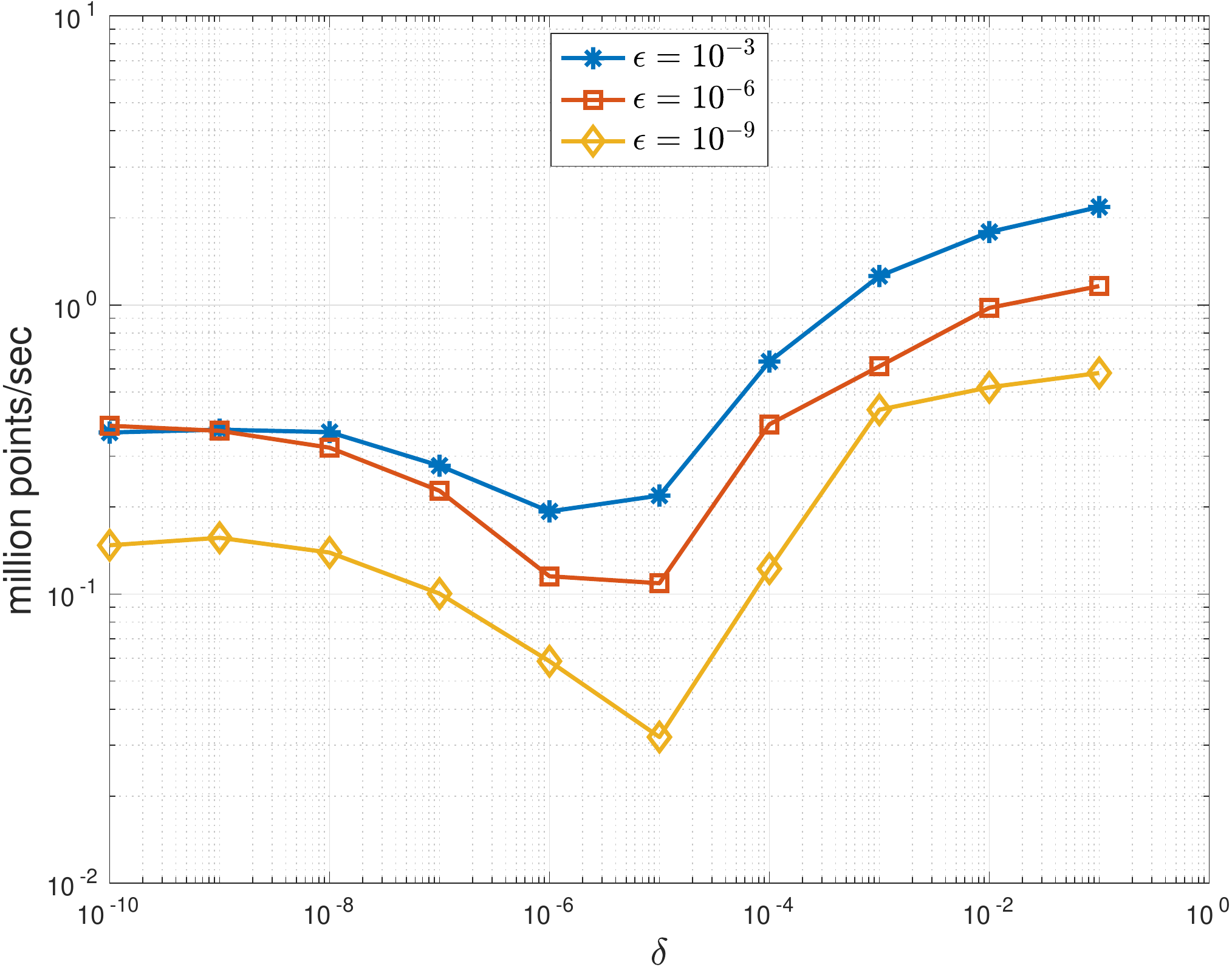}
\caption{\sf Performance of the 3D point FGT as a function of 
Gaussian variance, with throughput measured in units of
million points per second:
(left) a uniform distribution, (right)
a highly nonuniform distribution.}
\label{pfgt3d}
\end{figure}

\begin{remark}
The point FGT can be further accelerated - first by more careful application
of the NUFFT (optimizing the FFTW call to make sure it only
carries out the ``plan" phase once), and second by using special purpose
libraries that accelerate the evaluation of exponential functions. 
\end{remark}

\subsection{Box FGT experiments}

We now investigate the performance of the FGT for continuous sources 
(the ``box" FGT). In our experiments, we fix
the polynomial approximation order to be $k=16$, and construct the tree 
so that the input data is resolved to $10$ digits of accuracy on each leaf 
box.  The precomputation of transformation/translation matrices 
in \cref{alg4} takes
approximately $1$--$3$ milliseconds, a negligible cost for the problems
we consider here.

When considering periodic boundary conditions, we assume the density takes the
form
\be\label{sigmap}
\sigma_p(\x)=\prod_{\substack{i=1\\ i\, {\rm odd}}}^d\sin(2\pi n_p x_i)\prod_{\substack{i=2\\ i\, {\rm even}}}^d \cos(2\pi n_p x_i),
\quad \x\in\mathbb{R}^d.
\ee
Refinement in this case leads to a uniform tree (and the analytical expression for the potential is available by Fourier analysis).
For our free space examples, we define the density as the sum of several
sharply peaked Gaussians
\be\label{sigmaf}
\sigma_f(\x)=\sum_{i=1}^{n_f} e^{-\|\x-\x_i \|^2/\alpha_i}.
 \quad \x\in\mathbb{R}^d
\ee
For $d=3$, we let
\be\label{xi3d}
\ba
\left[\x_1,\ldots,\x_5\right] &= [(-0.3, -0.4, -0.06), (-0.2,0,-0.25),\\
  & (0.18, -0.1,-0.03),(-0.09, 0.3,0.17), (-0.38, -0.05,-0.17)], 
\ea
\ee
and for $d=2$, we let $\x_i$  be the projections of the points in \eqref{xi3d}
to the $x_1x_2$ plane. We specify the corresponding
$\alpha_i$'s below.
The expression for the potential, given the density~\eqref{sigmaf},
is easily obtained analytically.

We first demonstrate the linear scaling of the box FGT, fixing
$\delta=10^{-3}$, 
For the periodic box FGT in 2D, we use the density ~\eqref{sigmap} 
with $n_p=8,16,32,64$. 
For the free-space box FGT in 3D,
we use the density ~\eqref{sigmaf} and increase the number of 
Gaussians $n_g$ from $2$ to $5$
with $\alpha_1=10^{-2}/n_g$ and $\alpha_i=\alpha_1/i$ for $i=2,3,4,5$. 
This forces the use of more and more points to resolve the density
as $n_g$ increases.
\cref{bfgtN} shows our results, with the throughput measured
in million points per second for various precisions.
\begin{figure}[th]
\centering
\includegraphics[height=40mm]{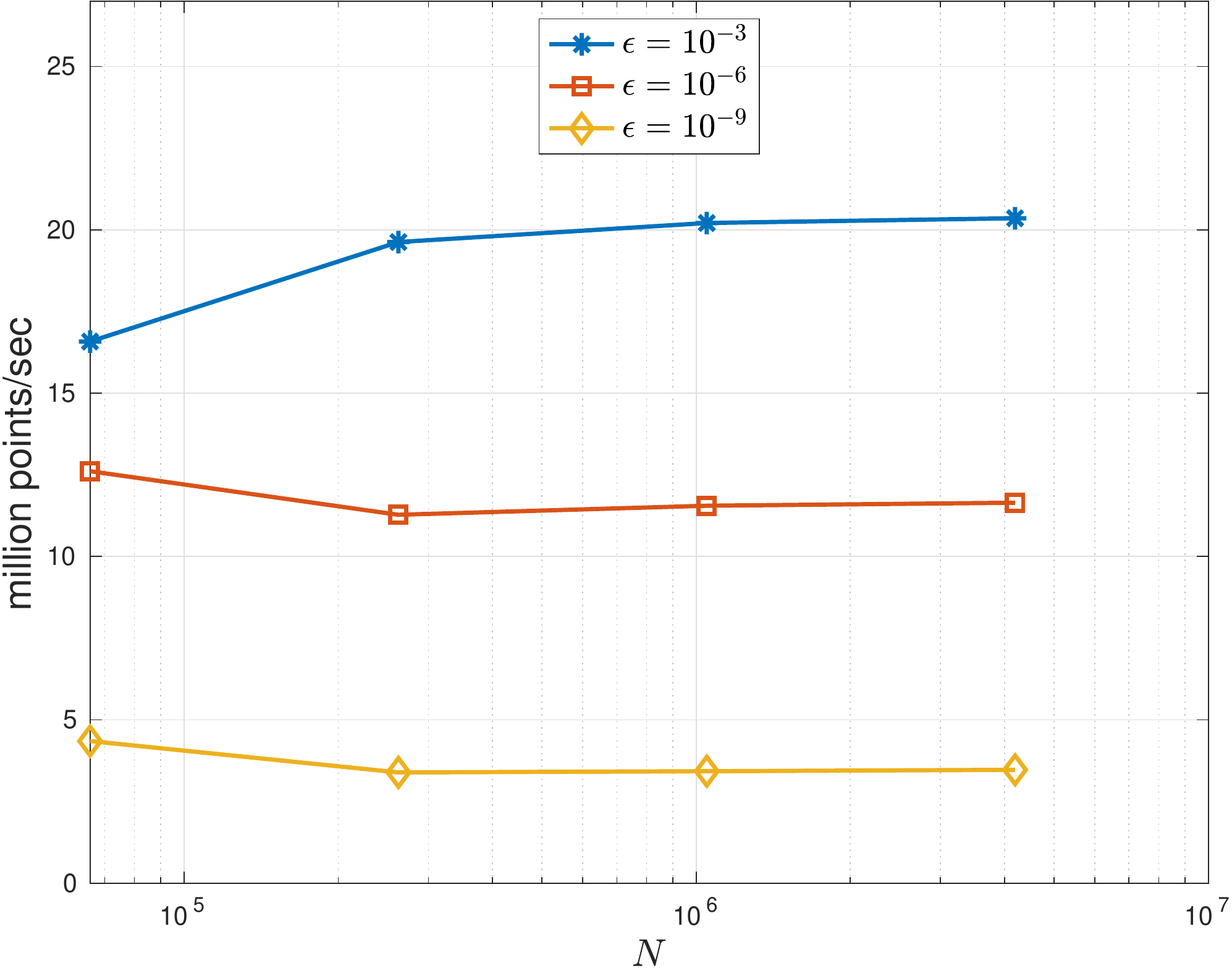}
\hspace{0.4in}
\includegraphics[height=40mm]{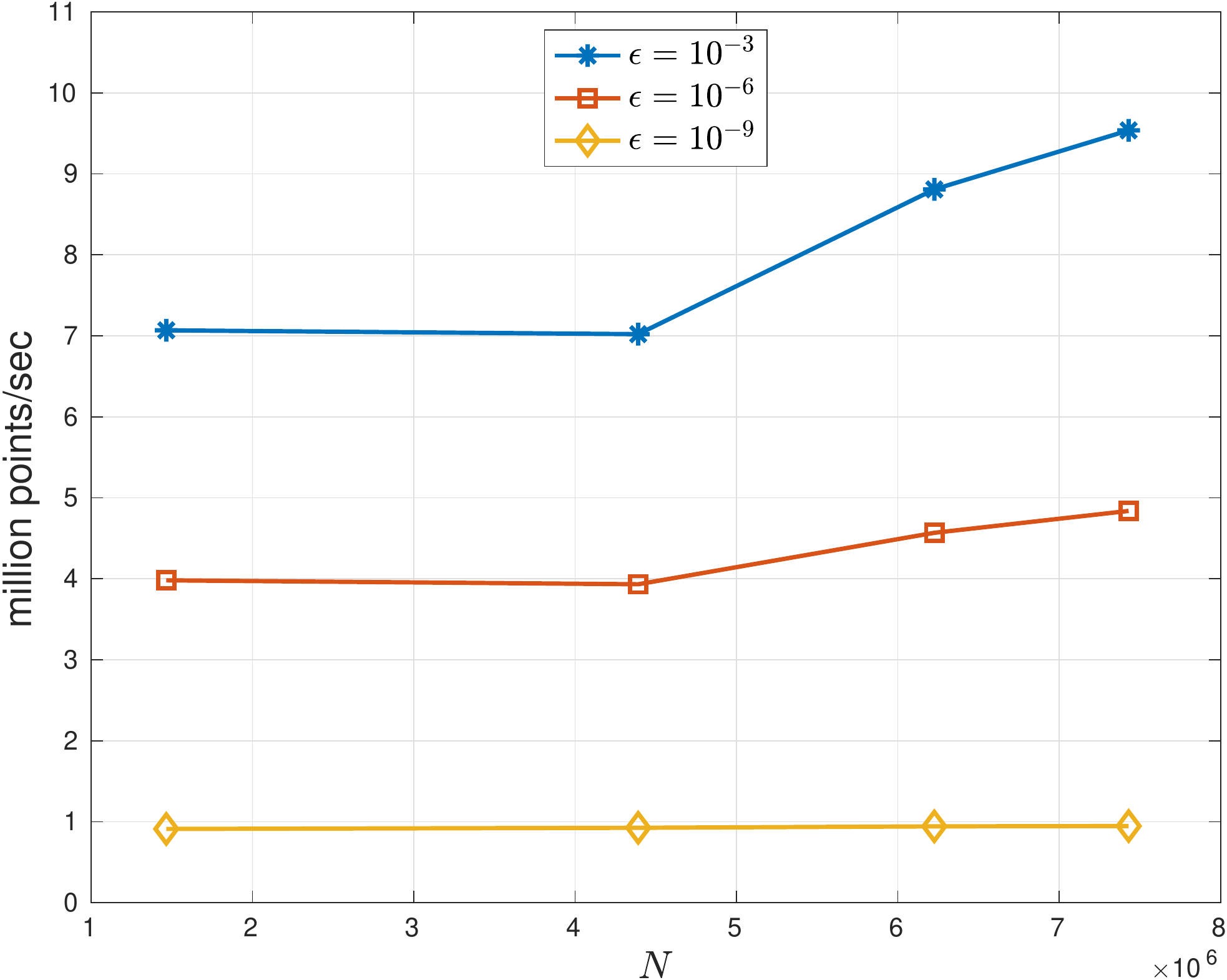}
\caption{\sf Linear scaling of the box FGT. 
Here, $N$ is the total number of tensor
product grid points in the tree. The $y$-axis indicates
the throughput of the box FGT measured
in units of million points/second:
(left) a 2D periodic example, (right) a 3D free-space example.
}
\label{bfgtN}
\end{figure}

To study the effect of the Gaussian variance $\delta$,
we compute the continuous periodic Gauss transform with
density given by~\eqref{sigmap}, where $n_p=32$ for $d=2$ 
and $n_p=4$ for $d=3$.
For the free-space case, we use the density \eqref{sigmaf} with $n_g=5$.
We set $\alpha_1=10^{-5}$ for $d=2$ and $\alpha_1=4\cdot 10^{-3}$ for $d=3$.
We set $\alpha_i=\alpha_1/i$ for $i=2,3,4,5$ for both $d=2$ and $d=3$.
We then vary $\delta$ from $10^{-10}$ to $1$. 
The timings are shown in \cref{bfgt2d} and \cref{bfgt3d} 
for $d=2$ and $d=3$, respectively.
Note that that the box FGT becomes very
fast for either very small or very large values of $\delta$. 
In the first case, all interactions are direct
and even the 1D transformation matrices $T_{pol\to pot}$ become sparser
and sparser as $\delta \rightarrow 0$. On the other hand,
when $\delta$ is very large, only a few plane-wave expansions are needed
for the entire computational domain.
\begin{figure}[th]
\centering
\includegraphics[height=40mm]{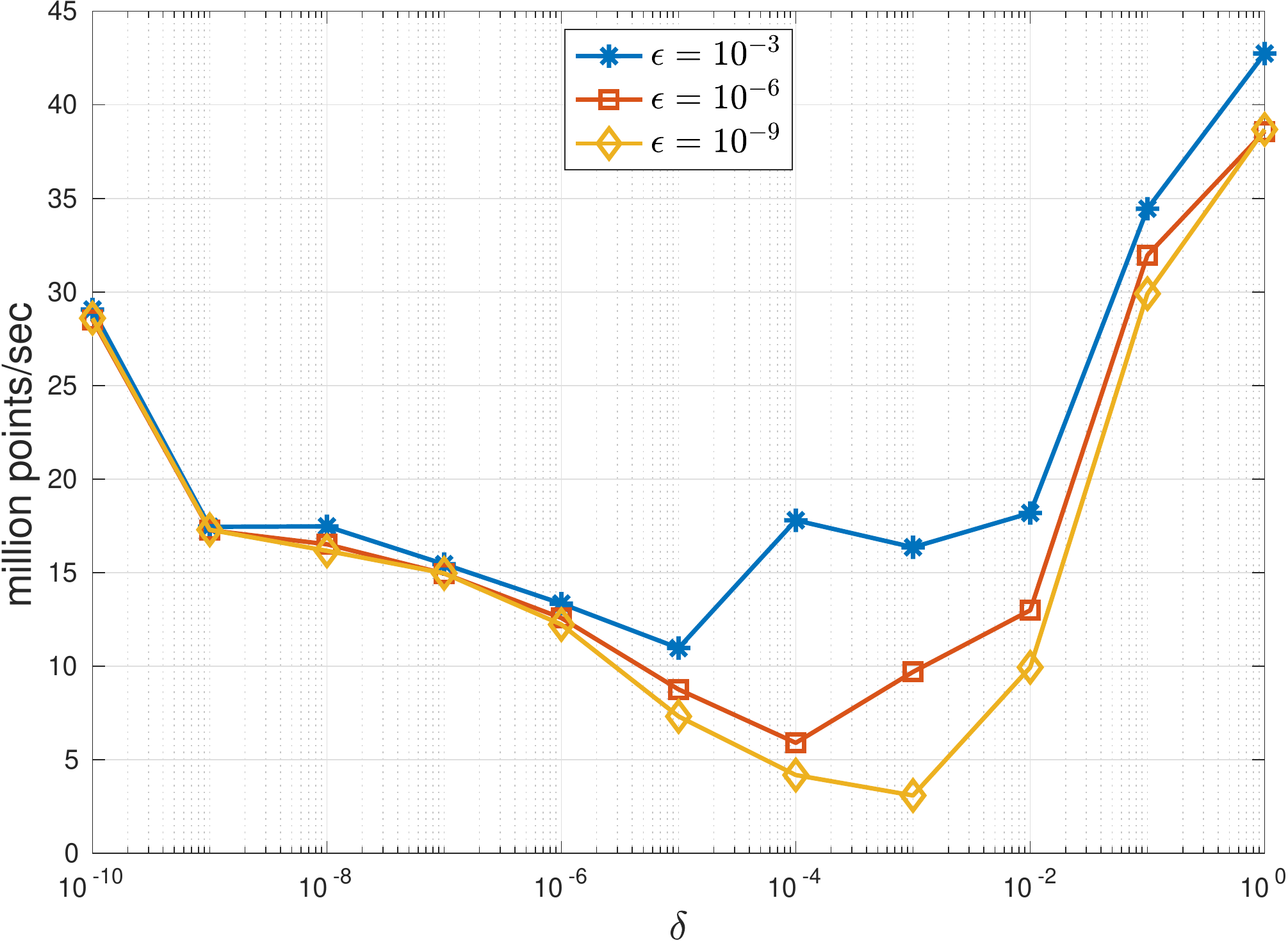}
\hspace{0.4in}
\includegraphics[height=40mm]{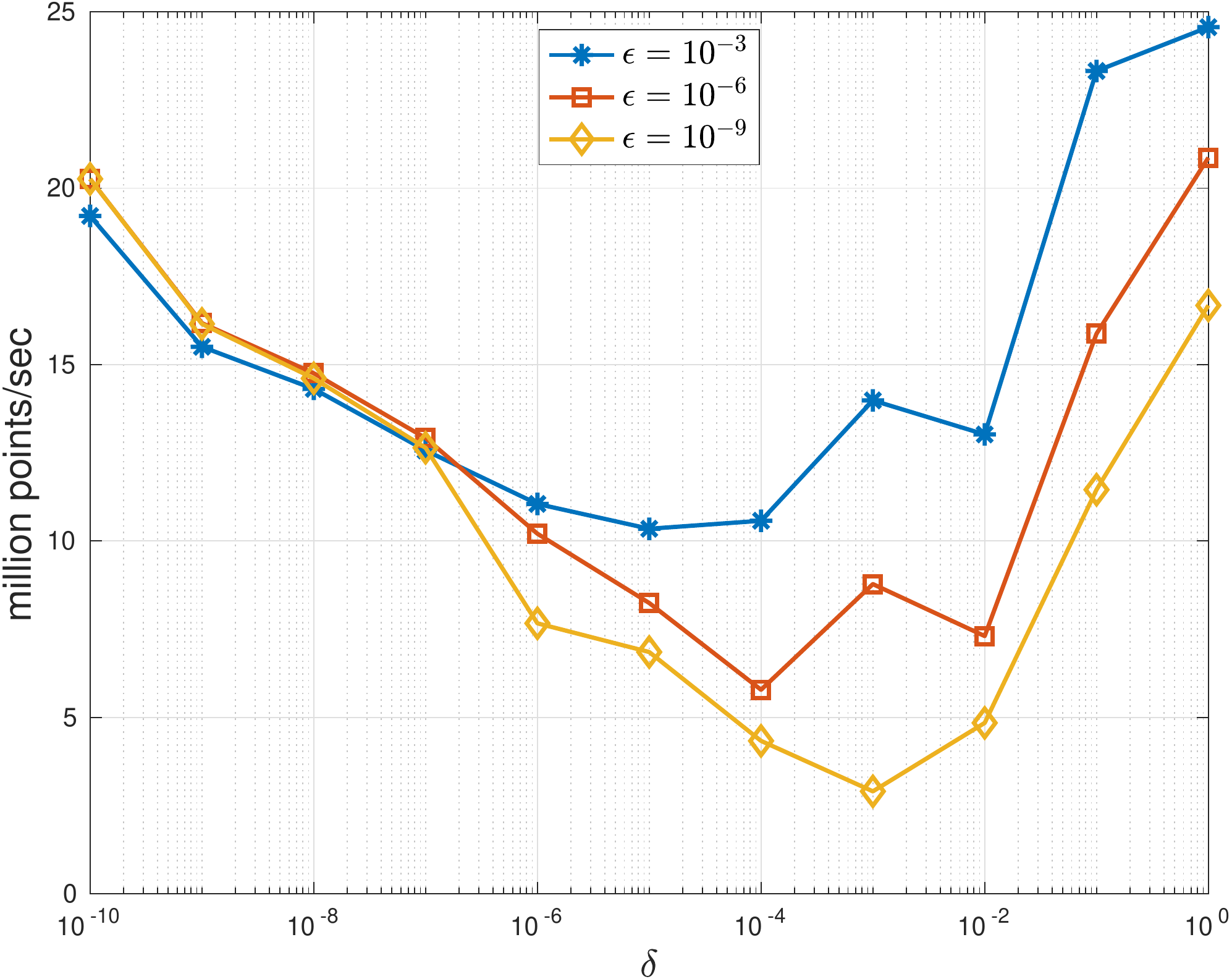}
\caption{\sf Performance of the 2D box FGT as a function of Gaussian variance:
  (left) periodic Gauss transform of the density \eqref{sigmap} with
  about one million points,
  (right) free-space Gauss transform of the density \eqref{sigmaf}
  with about a quarter of a million points.}
\label{bfgt2d}
\end{figure}
\begin{figure}[th]
\centering
\includegraphics[height=40mm]{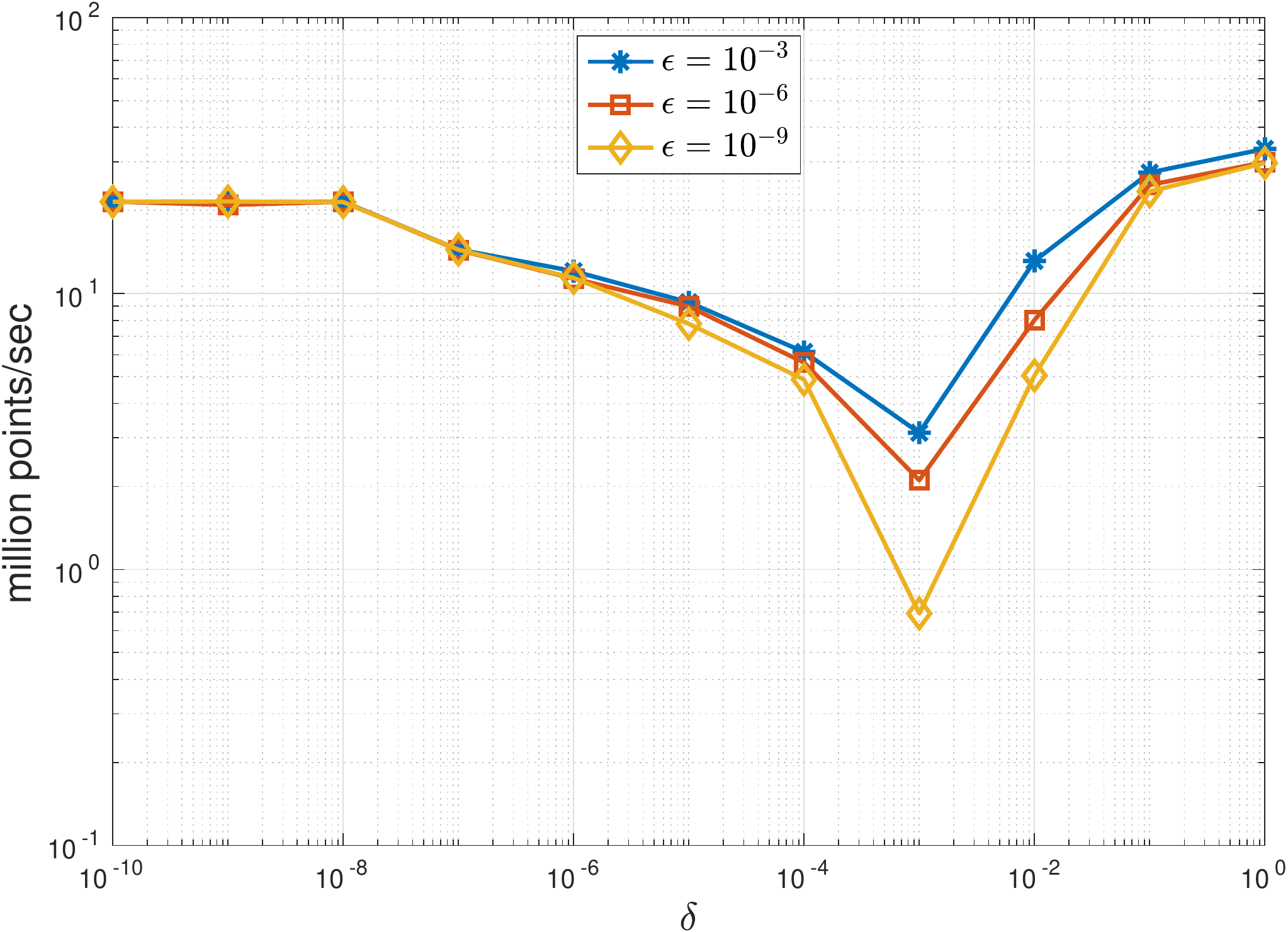}
\hspace{0.4in}
\includegraphics[height=40mm]{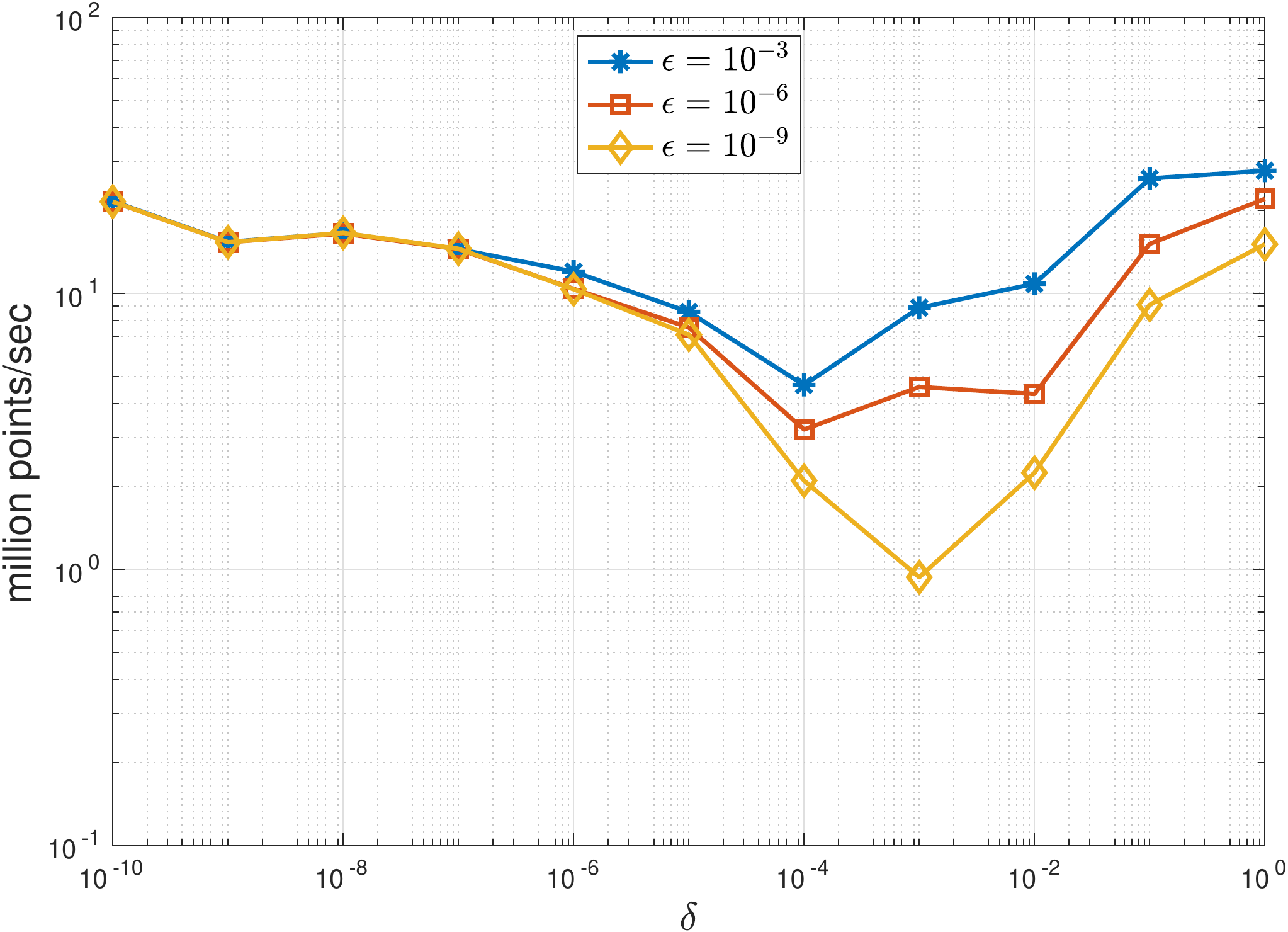}
\caption{\sf Performance of the 3D box FGT as a function of Gaussian variance:
  (left) periodic Gauss transform of the density \eqref{sigmap}
  with about two million points,
  (right) free-space Gauss transform of the density \eqref{sigmaf}
  with about six million points.}
\label{bfgt3d}
\end{figure}

\section{Conclusions and further discussions} \label{sec:conclusions}

We have developed a new version of the fast Gauss transform 
which combines the benefits of the single-level plane-wave based approach
of \cite{greengard1998nfgt,sampath2010pfgt,spivak2010sisc}
with hierarchical merging, as in 
\cite{wang2018sisc}. 
For discrete sources,
the scheme refines to a uniform cut-off level, and makes use of the NUFFT
to convert the source distribution in a cut-off level box to a plane-wave 
expansion. 
For continuous
sources, discretized on adaptive tensor product grids, separation of variables 
allows for the direct mapping of the source density to the 
plane-wave representation at even greater speeds.
We have limited ourselves to dimensions $d \leq 3$ but the algorithm could
be pushed a little further although it does not avoid the curse of 
dimensionality.
We did not consider one-dimensional examples in any detail, partly 
because ordering of the unknowns on the real line
permits the use of the SOE (sum-of-exponentials)-based scheme 
\cite{jiang2022cicp}. This is even faster for 
1D problems but, as it turns out, not as fast as the present scheme for
$d>1$.

There are a number of computational kernels which can be approximated
as a sum of Gaussians (see, for example, \cite{beylkin2005,beylkin2010}) and 
the FGT described here helps accelerate the corresponding integral
transforms. We are currently integrating the new FGT into solvers for 
time-dependent diffusion problems. We also plan to release an open source software
package for the new scheme in the near future. 

Finally, we note that the current implementation is not fully optimized for
modern CPUs. For example, no use is made of SIMD vectorization 
in any part of the code, including the FINUFFT library~\cite{finufftlib}.
It is, in fact, possible to remove the dependence on NUFFTs entirely, even in 
the point FGT, by replacing the original source distribution in a leaf node
with an equivalent tensor product
grid of ``proxy points," as in \cite{pfmm}. Transformation to the 
plane-wave basis can then be carried out in tensor product form
with level-3 BLAS acceleration. Such low-level modifications are likely
to increase the throughput of the point FGTs over the data presented in
\cref{sec:numer}, even though the reduction in terms of operation
count is small.

\bibliographystyle{siam}
\bibliography{journalnames,fgt}

\end{document}